\documentclass[reqno]{amsart}
\usepackage{amsfonts,amsthm,amsmath,amssymb,fullpage,graphicx,bm,xcolor,bbm}
\usepackage[utf8]{inputenc}

\usepackage[british]{babel}
\usepackage[shortlabels]{enumitem}
\usepackage[font=sf]{caption}

\allowdisplaybreaks

\usepackage{hyperref} 
\hypersetup{colorlinks=true, linkcolor=blue, citecolor=blue, filecolor=blue, urlcolor=blue}
\usepackage[numbers, comma, sort&compress]{natbib}

\usepackage[shortlabels]{enumitem}
\usepackage[comma,sort&compress]{natbib}

\definecolor{greennew}{rgb}{0.1
,0.4,0.1}
 \title{\bf{Branching random walk with infinite progeny mean: a tale of two 
tails}}

  \author{Souvik Ray}\thanks{Department of Statistics, Stanford University.  
souvikr@stanford.edu}
   \author{Rajat Subhra Hazra}\thanks{Leiden University, Netherlands and Indian Statistical Institute, Kolkata.  
rajatmaths@gmail.com}
    \author{Parthanil Roy}\thanks{Indian Statistical Institute, Bangalore.  
parthanil.roy@gmail.com}
    \author{ Philippe Soulier}\thanks{Universit\'e Paris-Nanterre.  
philippe.soulier@parisnanterre.fr}

\numberwithin{equation}{section}

\date{}
\usepackage{cleveref}
\newtheorem{theorem}{Theorem}[section]

\newtheorem{proposition}[theorem]{Proposition}
\newtheorem{corollary}[theorem]{Corollary}
\newtheorem{remark}[theorem]{Remark}
\newtheorem{lemma}[theorem]{Lemma}
\newtheorem{example}[theorem]{Example}
\newtheorem{assumption}[theorem]{Assumption}

\crefname{equation}{}{}
\crefname{theorem}{Theorem}{Theorems}
\crefname{assumption}{Assumption}{Assumptions}
\crefname{remark}{Remark}{Remarks}
\Crefname{lemma}{Lemma}{Lemmas}
\crefname{lemma}{Lemma}{Lemmas}
\crefname{enumi}{}{}

\begin{document}

\begin{abstract}
We study the extremes of branching random walks under the assumption that the
  underlying Galton-Watson tree has infinite progeny mean. It is assumed that the displacements are
  either regularly varying or they have lighter tails. In the regularly varying case, it is shown
  that the point process sequence of normalized extremes converges to a Poisson random measure. We study the asymptotics of the scaled position of the rightmost particle in the $n$-th generation when the tail of the displacement behaves like $\exp(-K(x))$, where either $K$ is a regularly varying function of index $r> 0$, or $K$ has an exponential growth. We identify the exact scaling of the maxima in all cases and show the existence of a non-trivial limit when $r> 1$.
\end{abstract}



\maketitle
\section{Introduction}
\label{sec:intro}

Branching random walk is a very important model in the context of statistical physics and
probability. The basic model is very simple and intuitive. It starts with a particle at the
origin. The particle splits into a random number of particles following a specified progeny
distribution and each new particle makes a random displacement on $\mathbb{R}$. The new particles
form the first generation. Each particle in the first generation splits into a random number of
particles according to the same law and independently of the past  as well as of the other particles
in the same generation. Each new particle makes a random displacement from the position of its
parent following the same displacement distribution, independently from other particles. The new
particles form the second generation. This mechanism  goes on. This resulting system is called a
branching random walk (BRW).

It is clear that the particles in the system described above naturally form a rooted Galton-Watson
tree if we forget about their positions. The progeny distribution of this branching process will be
denoted by $\left\{p_k \right\}_{k \geq 0}$, with $p_k := \mathbb{P}(Z_1 = k)$, where $Z_n$ denotes
the number of individuals at generation $n \geq 0$ with $Z_0\equiv 1$. This Galton-Watson tree
will be denoted by $\mathbb{T}=(V,E)$, where $V$ is the set of vertices of the tree and $E$ is the
collection of edges. The collection of particles or vertices at the $n$-th generation will be
denoted by $D_n$. 

We identify each edge $e_v$ of the Galton-Watson tree with its vertex $v$ away from the root; we then assign
a real-valued random variable $X_{e_v}$,  the displacement of the corresponding
particle. Our model implies that conditioned on the Galton-Watson tree $\mathbb{T}$,
$\{X_e: \, e\in E\}$ is a collection of i.i.d. random variables. Because of the underlying tree
structure, for each vertex $v$, there is a unique geodesic path connecting it to the root. We shall
denote the collection of all edges on this path by $I_v$. It is easy to see that the position of the
particle corresponding to the vertex $v$ is given for $v\in V$ by
$$ S_v := \sum_{e \in I_v} X_e.$$
The collection $\left\{S_v: v \in V \right\}$ is called the Branching Random Walk (BRW) induced by
the tree $\mathbb{T}=(V,E)$ and the displacements $\left\{X_e : e \in E \right\}$. The main focus of
the study of BRW is the study of the asymptotic behavior of $\left\{S_v : v \in D_n \right\}$ when
$n$ tends to $\infty$, or the behavior of functions such as the maximum displacement,
$M_n := \max_{v \in D_n} S_v$; the range of the displacements
$R_n := (\max_{v \in D_n} S_v - \min_{v \in D_n} S_v)$; order statistics or different gap
statistics, etc.

{\bf Literature review:}

The earliest works on branching random walks include \cite{Hammersley,Kingman,Biggins}. This model and its extreme value theory have now become very important because of
their connections to various probabilistic models (e.g., Gaussian free fields, conformal loop
ensembles, multiplicative cascades, tree polymers, etc.); see \cite{BramsonZeitouni,BerryReed, HuShi, aidekon, deywaymire}. Extremes of the branching random
walk with heavy-tailed displacement has been studied by \cite{durrett79, durrett83, gantert,
  bermal, maillard, abrshpr, bhattacharya2018large, bhatt2017}. The point process of displacements of a branching random walk with finite progeny mean is described
through a Cox-cluster process, rather implicitly in case of light-tailed displacements
(\cite{madaule}) and more explicitly in the heavy tailed set-up (\cite{bhatt2017}). The limiting
point process seems to have an universal stability structure, as was predicted by
\cite{brunetderrida}. For a detailed discussion of such stability properties, we refer the readers
to \cite{maillard:2013, subag2017extremal, bhattacharya2018note} and also refer to the exposition by
~\cite{shi2015} for a detailed background on the topic. Infinite mean branching processes and branching random walks (with infinite progeny mean) are intimately tied up with many scale free networks and hence important in study of random graphs. See for example the recent work of \cite{van2007distances,van2005distances,komjathy2019explosion, deijfen2013scale, van2017explosion} which explore the relationship of infinite mean branching process with various graph properties. The branching random walks considered in these random graph models live on $\mathbb Z^d$.

The main focus of this paper will be on the analysis of the behavior of the BRW when the progeny
distribution has infinite mean, i.e., $\mathbb{E}(Z_1) = \infty$. We would like to point out here the conditions in \cite{aidekon} under which the weak limit for the left-most position was computed. Although the conditions allowed for the progeny distribution to have infinite mean, they will fail to hold whenever the progeny variables have infinite mean and are independent from the tree structure, the foundational assumption of our analysis in this paper. In the branching process
literature, the asympototic behaviour of the number of particles in the $n$-th generation under
infinite mean was first studied in \cite{seneta1973, darling1970}. The conditions in
\cite{seneta1973} were later improved by  \cite{Davies, schuh:barbour, grey}. In this article we shall follow throughout
the sufficient conditions mentioned in \cite{Davies}.  It was shown in \cite{Davies} that if the progeny
distribution has a moment index (cf.~\cref{momentindex}) $\alpha\in(0,1)$, then
\begin{equation}
\label{daviesresultintro}
\alpha^n \log Z_n \stackrel{a.s}{\longrightarrow} W,
\end{equation}
where $W$ is a non-negative random variable, and almost surely positive on the event
of survival of the tree. In other words, in the infinite mean set-up, the generation size explodes
in a double-exponential manner if the tree survives. As a consequence we establish that in this case
the Galton-Watson tree, up to the $n$-th generation, has most of its particles in the last
generation, i.e., the total progeny up to the $(n-1)$-th generation is negligible when compared to
the number of particles at the $n$-th generation (see Lemma~\ref{conv1}). This presence of
a huge number of particles in the last generation, shows that most pairs of particles
in the last generation have very few common ancestors and therefore the dependence between their
displacements is very low. Consequently, it is expected that the behaviour of
$\left\{S_v : v \in D_n \right\}$ will be close to the behaviour of $Z_n$ many independent
realizations of the displacement random variable.  These heuristics will provide the correct results
when displacements are heavy tailed but not when the displacements are light tailed. We shall see
that in the case of distributions with tails decaying at infinity at a rate faster than the exponential distribution, contributions from other generations, and hence the appropriately scaled maxima converge to a non-trivial constant.

{\bf Main contributions:}

In Section~\ref{heavytail} we shall restrict our attention to the situation where the
  displacement distribution is almost surely non-negative and has a regularly varying tail 
  with index $-\beta$, i.e., if $F$ is the displacement distribution function, then
  \begin{equation}
    \label{def:RV}
    \lim_{x \to \infty} \dfrac{1-F(tx)}{1-F(x)} = t^{-\beta},\; \forall \, t >0,
  \end{equation}
  for some $\beta>0$. We denote the class of regularly varying functions with index $-\beta$ by $RV_{-\beta}$. To analyze the asymptotic behaviour of the $n$-th generation, as $n$ becomes
  large, we shall take the approach of point process theory.  We shall scale the positions of the
  particles in the $n$-th generation by $C_n:= F^{\leftarrow}(1-\frac{1}{Z_n})$. Throughout this paper, $F^{\leftarrow}$ will denote the (left-continuous) inverse of $F$, defined as $F^{\leftarrow}(y):=\inf \left\{ s : F(s) \geq y\right\}$, for all $y \in [0,1]$. Choice of this
  random scaling is inspired by the deterministic scaling used in \cite{abrshpr}. In
  Theorem~\ref{th:regvarmain}, we show that the point process converges to a Poisson random measure with
  intensity measure $\tau_\beta$, where $\tau_{\beta}([x,\infty))=x^{-\beta}$, for all $ x >0$.  This shows there is no
  clustering in the limit under the above scaling and the dependency structure gets camouflaged by
  the size of the last generation. An important consequence comes from Corollary~\ref{cor1}. We  show
  that the maximum displacement $M_n$ grows doubly exponentially conditioned on the survival of the tree, that is
  \begin{equation}
    \label{eq:intromax}
    \alpha^n \log M_n \overset{P}\rightarrow \frac{W}{\beta}.
  \end{equation}
 Here $\stackrel{P}{\longrightarrow}$ corresponds to convergence in probability. One can, in fact show almost sure convergence of the first $k$ order statistics in the log scale to the
  same limit, see \cref{aslim:regtail}. As announced, both tail indices come into play in these asymptotics. The point process
  result stated in \cref{th:regvarmain} can be used also to get various other order statistics of the displacement random
  variables.  The proof of the point process convergence relies on the one large jump principle
  which we use to show that the point process based on the scaled positions in the $n$-th generation
  is close (in an appropriate metric) to the point process based on the displacements in the last
  generation (see Lemma~\ref{conv1}) .

In \Cref{section:ltd} we consider the case when the right tail is no-longer regularly varying. We assume that tail of the displacement distribution is asymptotically like $\exp(-K(x))$ where $K(x)$ is regularly varying with index $r \in (0, \infty)$. Important examples of such distribution include Gaussian, exponential and Weibull random variables. Under some additional conditions on the left tail we show in Theorem~\ref{lighttailmain} that almost surely the following asymptotics is true:

\begin{equation}\label{eq:limitlight}
    \lim_{n\to\infty} \frac{M_n}{L(\log Z_n)}
   = \begin{cases} (1-\alpha^{\frac1{r-1}})^{\frac1r-1}, &\mbox{if } r > 1, \\
    1, & \mbox{if } 0<r \leq1. \end{cases}, 
\end{equation}
  where $L = K^{\leftarrow}$ is the left-continuous inverse of $K$. We would like to point out the change in behaviour of the maximum when $r>1$. When $r \leq 1$, it still happens that the resulting contribution comes from the last generation of the branching process but when $r>1$, there is contribution from all the generation and these contributions cannot be neglected. 
   Although this does not change the rate of growth for the maximum displacement, the effect is apparent in the limiting constant.

In  Section \ref{section:veryltd} we consider the case when the right tail decays much faster than those considered in  Section \ref{section:ltd}. In particular, we assume that $1-F(x)=\exp(-K(x))$ where $L:= K^{\leftarrow}$ is slowly varying at $\infty$. Under some additional technical conditions on the growth rate of $L$, we show in \cref{verylighttail:mainthm} that almost surely the following asymptotics hold true:
\begin{equation}\label{eq:limitverylight}
    \lim_{n\to\infty} \frac{M_n}{\sum_{k=1}^n L(\alpha^{-k})}
   =1.
\end{equation}
Unlike the previous cases, here each generation contributes equally to the right-most displacement in the last generation and the resulting maximum position $M_n$ has magnitude of strictly higher order than $L(\log Z_n)$, the magnitude of the largest displacement in the last generation. See \Cref{section:veryltd} for a detailed explanation of the phenomenon.

It is noteworthy that the results for the displacement in the Gumbel domain of attraction are not uniform, contrary to the i.i.d. case. Here, we observe different normalizations according to the tail, whereas for $n$ i.i.d. observations, it always hold that $M_n/L(\log n)$ converges to $1$ in probability, and almost surely under a very mild restriction, see  \cite[Theorem 5.4.5]{haan2006}.

The speed of a branching random walk can be defined in many ways. The cloud speed, burst speed and sustainable speed are some of the
possibilities. We refer to
\cite{peres} for definition and detailed discussions on these topics. The equivalence of these three notions of speed was established by \cite{Hammersley, Kingman,Biggins} under assumptions of
finite progeny mean and finiteness of the moment generating function. The later condition was
removed by \cite{gantert} where the tail of displacement random variables were assumed to follow a
semi-exponential distribution, which changed the rate of growth of the maxima. The definitions of the speeds was modified accordingly. In \Cref{sec:speed}, we define properly scaled versions of cloud, burst and sustainable speeds under \Cref{assumption:displacement_heavy}, \Cref{assumpF} and \Cref{verylighttail:assump}. We then establish the equivalence of these three notions of speeds in each of those cases.

%

We would now like to point out an aspect  which differentiates the nature of  the main results in \Cref{heavytail}, \Cref{section:ltd} and \Cref{section:veryltd}. In \Cref{heavytail}, \cref{cor1} provides us with an weak limit of properly scaled $M_n$ (as stated in \cref{probconv:reg}). This can be watered down to a version concerned with in-probability limit for $\log M_n$ after proper scaling, see \cref{reg:inprob}. Since the limit in \cref{reg:inprob} is non-degenerate, it can be viewed also as a weak limit. A similar phenomenon occurs in \Cref{lighttailmain} where we state an almost sure limit for properly scaled version of $M_n$ in \cref{eq:deterministic}, which again can be interpreted as an weak limit due to non-degeneracy of the limit. In both of these cases we take another logarithm and compute a degenerate almost sure limit which we refer to as the cloud speed.  In contrast to these, \cref{verylighttail:mainthm} immediately gives us a degenerate almost sure limit. In certain spirit, therefore, \cref{verylighttail:mainthm} provides us with the cloud speed when $F$ satisfies \cref{verylighttail:assump}. We do not get any weak limit here, although the proof technique is very similar to the one employed to prove \cref{lighttailmain}. 

\subsection*{Notation}
We use the notation $X_n\stackrel{a.s.}{\sim } a_n$ to indicate that $X_n/a_n\to 1$ almost surely. We also use the notion that $a_n$ grows at least exponentially to mean $\liminf_{n \to \infty} n^{-1}\log a_n >0.$ Similarly a sequence $a_n$ grows at least double-exponentially when $\liminf_{n \to \infty} n^{-1}\log \log a_n >0.$  Most of our results will be ``\emph{conditioned on the survival of the Galton-Watson tree}". In places where it is obvious we skip this phrase. Given the tree, we will denote $D_n$ to be the particles in the $n$-th generation. $e_v$ will denote the edge that connects particle or node $v$ to $p(v)$ where $p(v)$ will refer to the parent/immediate ancestor of the particle $v$. Also $C(v)$ will denote the set of all children of the particle $v$. Throughout the paper, the notations $\stackrel{a.s.}{\longrightarrow}, \stackrel{P}{\longrightarrow}$ and $\stackrel{D}{\longrightarrow}$ will stand for almost sure convergence, convergence in probability and weak convergence respectively. The notation $\mathbbm{1}(A)$ will denote the indicator function for the event $A$. $\text{Binomial}(n,p)$ and $\text{Poisson}(\lambda)$ will denote Binomial and Poisson distributions respectively with corresponding parameters. We also introduce the notation $G_{(\delta)}$ to denote the distribution function of $\lfloor Z^{\delta} \rfloor$ where $Z \sim G$, for any distribution function $G$ supported on non-negative integers.

\subsection*{Acknowledgement} The authors would like to thank Remco van der Hofstad for first
pointing out the article by \cite{Davies} and also many useful discussions throughout the
project. We also thank Ayan Bhattacharya for helpful comments during the progress of the article. A
significant portion of the research was carried out 
the visit of P.R.~to Universit\'{e} Paris Nanterre, and the
visit of P.S.~to Indian Statistical Institute, Bangalore. We acknowledge all of these institutions
for their support and hospitality. The research of R.S.H.~and P.R.~are both partially funded by
MATRICS Grants from the Science and Engineering Research Board, Govt.~of India. P.R.~is also
partially supported by a SwarnaJayanti Fellowship from the Department of Science and Technology,
Govt.~of India. 

We thank two anonymous referees for their helpful suggestions and comments. We also thank the referee for asking the question on the rapidly varying tails which this lead to some new results in Section \ref{section:veryltd}.

\section{Infinite progeny mean branching process}
\label{infmeanbp}
We shall first recall the main result of \cite{Davies} on the asymptotic properties of the
Galton-Watson tree under the assumption that the progeny mean is infinite. Throughout this paper,
$G$ will denote the distribution function of $Z_1$, the non-negative integer valued branching progeny and $\bar{G} = 1-G$ will denote its survival function. 
\begin{assumption} [{\bf Assumption on the branching random variable}]
  \label{davies_assump}
  There exists a function $\gamma : \mathbb{R}^{+} \to \mathbb{R}^{+}$ and a constant
  $ \alpha \in (0,1)$ such that,
  \begin{enumerate}[leftmargin=*,label=(D\arabic*)]
  \item\label{item:D1}$\gamma$ is non-increasing.
  \item\label{item:D2} $x \mapsto x^{\gamma(x)}$ is non-decreasing.
  \item\label{item:D3} $\int_{0}^{\infty} \gamma(e^{e^x})\;dx < \infty.$
  \item\label{item:D4} $\exists \; x_0 >0$, such that
    \begin{align*}
    x^{-\gamma(x)} \leq x^{\alpha} \bar{G}(x) \leq x^{\gamma(x)}, \;\; \forall \; x \geq x_0.
    \end{align*}
  \end{enumerate}
\end{assumption}
\begin{remark}{\label{momentindex}}
It is easily seen that $\alpha$ is the moment index of $G$, i.e.~$\mathbb{E}[Z_1^p]<\infty$ for
  all $p<\alpha$ and $\mathbb{E}[Z_1^p]=\infty$ for all $p>\alpha$.
Distribution functions with Pareto tails of the form $cx^{-\alpha}$ satisfy
  \cref{davies_assump}. More generally, if $\bar{G}(x)=x^{-\alpha}L(x)$ with $L$ slowly varying and
  either $(\log L(x))/\log(x)$ or $-(\log L(x))/\log(x)$ satisfying
  \cref{item:D1,item:D2,item:D3,item:D4}, then $G$ satisfies \cref{davies_assump}. However, regular
  variation of $\bar{G}$ is not implied by \cref{davies_assump}. It is also easily seen that if two
  distributions $G_1$, $G_2$ are tail equivalent (in the sense that there exist finite, positive
  constants $C_1$ and $C_2$ such that
  $C_1\bar G_1(x) \leq \bar G_2(x) \leq C_2\bar G_1(x)$ for large $x$) and one
  satisfies the conditions in \cref{davies_assump}, then so does the other one. 
  
  \end{remark}



We now quote the main result of \cite{Davies}. 
\begin{theorem}[Theorem 1, \cite{Davies}]
  \label{davies}
  Under \cref{davies_assump}, there exists a non degenerate, non-negative random variable
  $W$ such that
  \begin{equation}
    \label{eq:davies:conv}
    \alpha^n \log(Z_n +1) \stackrel{a.s.}{\longrightarrow} W. 
  \end{equation}
  Moreover, 
  $\mathbb{P}(W=0)=q$ where $q$ is the probability of extinction of the Galton-Watson tree.
\end{theorem}

The convergence \cref{eq:davies:conv} shows that conditioning on the survival of the tree is
equivalent to conditioning on the event $W>0$,
that is, 
the events $\{\mathbb{T} \; \text{survives}\}$ and $\{W>0\}$ differ by an event of probability zero.

A consequence of \cref{davies} which will be used to prove our results is the following lemma which
tells that almost all the mass of the tree is concentrated in the last generation. To be more precise, total mass of the tree before the last generation is comparable to the mass of the last generation only in the log-scale.  
\begin{lemma}
  \label{conv1}
  Assume the progeny distribution satisfies \cref{davies_assump}. Then for any $s>0$,
  conditionally on survival ot $\mathbb{T}$,
  \begin{align*}
    \frac{1}{\log Z_n} \log \left( \sum_{i=0}^{n-1}  Z_i^s \right) & \stackrel{a.s.}{\longrightarrow} s \alpha \; .
  \end{align*}
\end{lemma}
\begin{proof}
  Take
  $\omega \in (W>0) \cap (\alpha^n\log(Z_n+1) \rightarrow W)$. Then we have
  $\alpha^n \log(Z_n(\omega)) \to W(\omega)$. Choose $\varepsilon >0$ and get $n_0 \in \mathbb{N}$
  such that
  \begin{equation}
    \label{eq:Z_nbound_1}
    \exp(\alpha^{-n}(1-\varepsilon)W(\omega)) \leq Z_n(\omega) \leq \exp(\alpha^{-n}(1+\varepsilon)W(\omega)),\;\; \forall \; n \geq n_0.  
  \end{equation}
  Now, for all $n > n_0$, using~\eqref{eq:Z_nbound_1} we have
  \begin{align}
    \frac{1}{ \log Z_n(\omega)} \log \sum_{i=n_0}^{n-1}  Z_i^s(\omega)
    &\leq  \frac{1}{\log Z_n(\omega)} \log \sum_{i=n_0}^{n-1}\exp(\alpha^{-i}(1+\varepsilon)sW(\omega))\nonumber \\
    & \leq \frac{1}{\log Z_n(\omega)} \log \left[n\exp(\alpha^{-(n-1)}(1+\varepsilon)sW(\omega)) \right] \nonumber\\
    &\leq \frac{\log n + \alpha^{-(n-1)}(1+\varepsilon)sW(\omega)}{\alpha^{-n}(1-\varepsilon)W(\omega)}
      \longrightarrow  s\alpha \frac{1+\varepsilon}{1-\varepsilon}\label{eq:Z_nbound_2},
  \end{align}
  where the last line is true since $W(w)>0$. Noting that $Z_n(\omega) \longrightarrow \infty$ and then
  taking $\varepsilon \downarrow 0$, we conclude that
  \begin{align*}
    \limsup_{n \longrightarrow \infty} \frac{1}{ \log Z_n(\omega)} \log \sum_{i=0}^{n-1}  Z_i^s(\omega) \leq s\alpha \; .
  \end{align*}
  An exactly similar argument shows the lower bound. This completes the proof. 
\end{proof}

 Both \cref{davies} and \cref{conv1} consider the a homogeneous branching process, i.e.~a branching process with identical progeny distribution over the generations. Indeed, in this article we shall restrict our attention to homogeneous branching processes only. But the techniques we shall apply to prove the main results in  Section \ref{section:veryltd} will require some results analogous to \cref{davies}, but in the context of a special kind of inhomogeneous branching trees. We shall explain the premise and state the result below while deferring the proof to   Section \ref{app:prooflems}. 

Consider the inhomogeneous  branching process starting with one particle at the $0$-th generation and where the particles of the $n$-th generation, for any $n \geq 0$,  produce i.i.d. number of off-springs, having distribution function $G_{n}$, independent of previous generations as well as the particles in the same generation. We assume the following assumption on the progeny sequence $\left\{G_n : n \geq 0\right\}$. 
   
\begin{assumption}{\label{assump:inhom}}
There exists $x_0 \in (0, \infty)$ and a sequence of positive real numbers $\left\{\alpha_n : n \geq 0\right\}$ in $(0,1)$, bounded away from $1$, such that for some $\gamma : \mathbb{R}^{+} \mapsto \mathbb{R}^+$ satisfying \cref{item:D1}, \cref{item:D2} and \cref{item:D3}, we have the following.
$$ x^{-\gamma(x)} \leq x^{\alpha_n} \bar{G}_n(x) \leq x^{\gamma(x)}, \;\; \forall \; x \geq x_0, \; \forall \; n \geq 0.$$
\end{assumption}

\begin{theorem}{\label{inhom}}
 Let $Z_n$ denote the size of the $n$-th generation for the inhomogeneous branching process with progeny sequence $\left\{G_n : n \geq 0\right\}$ satisfying \cref{assump:inhom}. Then there exists an event $E$ with $\mathbb{P}(E)>0$ and a positive random variable $W^*$ such that 
$$  \left(\prod_{j=0}^{n-1}  \alpha_j \right) \log (Z_n+1) \stackrel{a.s.}{\longrightarrow} W^{*}, \text{on } E.$$
\end{theorem}
\begin{remark}
\cref{davies} is clearly a special case of \cref{inhom} for the choice of the sequence $G_n =G$ and $\alpha_n=\alpha$ for all $n \geq 0$. We want to emphasize the fact that the proof of \cref{inhom} uses \cref{davies} and hence does not provide an independent proof of \cref{davies}.
\end{remark}

\begin{remark}
The proof of \cref{inhom} uses \cref{strongheavy1}, which is d proved in \cref{app:prooflems}. We would like to point out here that \cref{strongheavy1} will also be instrumental in proving the main result in \cref{section:ltd}, namely \cref{lighttailmain}.
\end{remark}



\section{BRW with regularly varying displacements}
\label{heavytail}
In this section we shall describe the extremes of the branching random walk when the displacement
variables associated with the edges are i.i.d.~with regularly varying tail. In this case, one can
derive the exact asymptotics of the point process of rescaled positions and show that the behaviour
is similar to an i.i.d.~set-up. When the progeny distribution has finite mean and satisfies the
Kesten-Stigum condition the point process behaviour was described in~\cite{bhatt2017}. The extremes
in such a set-up (with finite mean) with more general conditions were derived
in~\cite{durrett83}. We now extend the above results to infinite mean progeny distribution.
\begin{assumption}
  \label{assumption:displacement_heavy}
  Given the tree $\mathbb{T}=(V, E)$ we assume 
  \begin{enumerate}[leftmargin=*,label=(R\arabic*)]
  \item\label{item:R1} the displacement random variables 
    $\left\{X_e\right\}_{e \in E}$ are {i.i.d.} with distribution $F$.
  \item\label{item:R2} the displacements are non-negative with probability $1$.
  \item\label{item:R3} $1-F \in RV_{-\beta}$, for some $\beta >0$.
  \end{enumerate}
\end{assumption}
The assumption~\ref{item:R2} can be replaced by two sided tail-balance condition and this will not
effect the analysis which follows and to keep the presentation simple we will stick to non-negative
random variables. We will later see that such a generalization when the displacements are
light-tailed will require more efforts. Recall that $F^{\leftarrow}$ is the left-continuous inverse
of $F$. Let us now define the random scaling
\begin{equation}
  \label{eq:scaling_heavy}
  C_n 
  := F^{\leftarrow}\left(1-\frac{1}{Z_n}\right) =\left(\frac{1}{1-F}\right)^{\leftarrow}\left(Z_n\right).
\end{equation}
Let us consider the set $\mathcal{E} := (0, \infty]$, with the usual topology (obtained by the
one-point uncompactification of $[0,\infty]$).  Keeping in mind the notation for the BRW defined in
\cref{sec:intro}, we define,
\begin{align*}
  N_n = \sum_{v \in D_n} \delta_{C_n^{-1}S_v} = \sum_{v \in D_n} \delta_{C_n^{-1}\sum_{e \in I_v} X_e}, \; \forall \; n \geq 1.
\end{align*}
We consider point processes as random elements in the space $M_p(\mathcal E)$ of all Radon point
measures on a locally compact and separable metric space $\mathcal E$ (for this section
$\mathcal E=(0,\infty]$). Here $M_p(\mathcal E)$ is endowed with the vague convergence; for further
details on point processes, see \cite{Resnick, resnick:2007, kallenberg:1986}.  The following result
describes the asymptotic behavior of the point process $N_n$. In this paper, $PRM(\mu)$, for a measure $\mu$ on $\mathbb{R}$, will refer to the \textit{Poisson random measure} on $\mathbb{R}$ with mean measure $\mu$; see \cite[Section 3.3]{Resnick} for a detailed exposition on Poisson random measures.

\begin{theorem}
  \label{th:regvarmain}
  Under \cref{davies_assump,assumption:displacement_heavy} we have
  \begin{align*}
    N_n
    \stackrel{D}{\longrightarrow} N \sim PRM(\tau_{\beta}), \;\; \text{ conditioned on survival},
  \end{align*}
  where
  $\tau_{\beta}$ is the unique measure on $(0, \infty]$ such that
  $\tau_{\beta}((x, \infty])= x^{-\beta}$.
\end{theorem}

The above result verifies \cite{brunetderrida} conjectures in this setup with the limiting extremal
point process being a Poisson random measure with no clustering. The reason behind this
cluster-breaking phenomenon is that the extremes are governed by the last generation displacements
thanks to regular variation and \cref{conv1}, which will be the key ingredient in the proof.

\begin{proof}[Proof of \cref{th:regvarmain}]
  The proof will consist of two steps. Let us first define
  \begin{align*}
    \widetilde{N}_n:= \sum_{v \in D_n} \delta_{C_n^{-1}X_{e_v}}, \;\forall \; n \geq 1,
  \end{align*}
  where $e_v$ denotes the edge that connects particle or node $v$ to its parent, for
  $v \in D_n, \; n \geq 1$. In the first step, we will prove that 
  \begin{align}
    \label{eq:first-step}
     \widetilde{N}_n \stackrel{D}{\longrightarrow} PRM(\tau_{\beta}) \; , 
  \end{align}
   conditionally on survival.  Let $d_v$ be a metric which induces the topology of
  vague convergence.  Next we shall demonstrate that the point processes $N_n$ and $\widetilde{N}_n$
  are asymptotically close, i.e.,
  \begin{align}
    \label{eq:second-step}
    d_v(\widetilde{N}_n, N_n) \stackrel{P}{\longrightarrow} 0  \; , 
  \end{align}
   conditionally on survival.  \cref{th:regvarmain} then follows by applying
  \cite[Theorem~4.1]{billingsley:1968}.
\end{proof}
	
\begin{proof} [Proof of \cref{eq:first-step}]
  Consider $\left\{X_i\right\}_{i \geq 1} \stackrel{i.i.d.}{\sim} F$, independent of
  $\mathbb{T}$. We define $b_n = F^{\leftarrow}(1-1/n)$. Thus $C_n = b_{Z_n}$ and
  \begin{eqnarray*}
    \widetilde{N}_n = \sum_{v \in D_n} \delta_{C_n^{-1} X_{e_v}}\stackrel{d}{=} \sum_{j=1}^{Z_n} \delta_{C_n^{-1}X_j}= \sum_{j=1}^{Z_n} \delta_{b_{Z_n}^{-1}X_j}
  \end{eqnarray*}
  By \cite[Proposition 3.21]{Resnick} we have
  \begin{align*}
    \sum_{i=1}^n \delta_{b_n^{-1}X_i} \stackrel{D}{\longrightarrow} PRM(\tau_{\beta}), \;\; \text{as \;} n \to \infty.
  \end{align*}  
  We also have $Z_n \stackrel{a.s.}{\longrightarrow}\infty$, conditioned on survival with
  $\left\{Z_i\right\}_{i \geq 1}$ being independent to $\left\{X_i\right\}_{i \geq 1}$. Therefore,
  \begin{align*}
    \widetilde{N}_n = \sum_{j=1}^{Z_n} \delta_{b_{Z_n}^{-1}X_j} \stackrel{D}{\longrightarrow}
    PRM(\tau_{\beta}), \;\; \text{as \;} n \to \infty, \text{conditioned on survival}.
  \end{align*}
\end{proof}

\begin{proof}[Proof of \cref{eq:second-step}]
  Here $\mathbbm{1}(\cdot)$ denotes the indicator variable for the event inside the parenthesis and  $A$ denotes the event that $\mathbb{T}$ survives. Fix Lipschitz continuous function
  $g \in C_K^{+}(\mathcal{E})$ with $\operatorname{supp}(g) \subseteq (\delta, \infty]$. It is
  enough to show that for any $\varepsilon >0$,
  \begin{align*}
   \mathbb{P} \left[|N_n(g) - \widetilde{N}_n(g)| \geq \varepsilon \Big{|}A \right] \longrightarrow 0.
  \end{align*} 
  Define $ U_{n} = \sum_{k=1}^{n-1}\sum_{v \in D_k} \mathbbm{1}\{C_n^{-1}X_{e_v} > n^{-2}\}$.  On the event
  $\{U_n =0\}$, all the displacements occurred until the $(n-1)$-th generation are of size at
  most $C_n/n^2$. Therefore, for any $v \in D_n$, we have
  \begin{align*}
    C_n^{-1}S_v-C_n^{-1}X_{e_v} = \sum_{e \in I_v : e \neq e_v} C_n^{-1} X_e 
    \leq \sum_{e \in I_v : e \neq e_v} n^{-2} = (n-1)n^{-2} < 1/n.
  \end{align*}
  Two situations can arise on $\{U_n=0\}$: if $C_n^{-1}X_{e_v} \leq \delta/2$, and $n$ is large
  enough such that $\delta/2 > 1/n$, then $C_n^{-1}S_v< \delta$, and therefore
  $g(C_n^{-1}S_v)=g(C_n^{-1}X_{e_v})=0$. Otherwise, using the Lipschitz continuity of $g$, we can
  say that $|g(C_n^{-1}S_v)-g(C_n^{-1}X_{e_v})| \leq M/n$, for Lipschitz constant
  $M \in (0, \infty)$. Combining these two bounds we conclude that, on $\{U_n=0\}$, for large enough
  $n$,
  \begin{align*}
    |N_n(g) - \widetilde{N}_n(g)| \leq \sum_{v \in D_n} |g(C_n^{-1}S_v)-g(C_n^{-1}X_{e_v})| \leq
    \frac{M}{n} \sum_{v \in D_n} \mathbbm{1} \left\{ C_n^{-1}X_{e_v} > \delta/2\right\} = \frac{M}{n}
    \widetilde{N}_{n}\left( (\delta/2, \infty]\right).
  \end{align*}
  Hence, for any $\varepsilon>0$, and large enough $n$,
  \begin{align}
    \label{eq:twoterms}
    \mathbb{P} \left[|N_n(g) - \widetilde{N}_n(g)| \geq \varepsilon \Big{|}A \right] \leq \mathbb{P}
    \left[ \frac{M}{n}\widetilde{N}_{n}\left( (\delta/2, \infty]\right) \geq \varepsilon \Big{|}A
    \right] + \mathbb{P}\left(U_n \geq 1 \Big{|} A \right).
  \end{align}
  By \cref{eq:first-step},
  $\{\widetilde{N}_{n}( (\delta/2, \infty]) \}_{n \geq 1}$ is a tight family
  since the set $(\delta/2,\infty]$ is relatively compact in $\mathcal{E}$. This implies that the
  first term on the right hand side converges to $0$.  We now prove that the second term in the right of
  \cref{eq:twoterms} tends to 0.  For any $e \in E$, conditioned on $\mathbb{T}$,
  $\mathbbm{1} \{C_n^{-1}X_e > n^{-2}\}$ is a Bernoulli random variable with parameter
  $p_n:= \bar{F}(n^{-2}C_n)$. Hence,
  \begin{equation*}
    \mathbb{P}\left[U_n \geq 1 \mid \mathbb{T}\right] \leq \mathbb{E} \left[U_n \mid \mathbb{T} \right] = p_n \sum_{k=1}^{n-1} Z_k.   
  \end{equation*}
  By Assumption~\ref{item:R3} and applying \cite[Proposition~0.8(i),(v)]{Resnick} on the
  non-decreasing function $1/\bar{F} \in RV_{\beta}$, we have
  $ \log F^{\leftarrow}(1-1/x) = (1/\beta + o(1))\log x$, as $x \to \infty.$ Therefore, on $(W>0)$,
  $\log C_n = \log F^{\leftarrow}(1-1/Z_n) = (1/\beta + o_p(1))\log Z_n $, which implies
  $C_n/n^2 \to \infty$. Applying again \cite[Proposition~0.8(i)]{Resnick} on the function
  $\bar{F} \in RV_{-\beta}$, we obtain
  $\log p_n=\log \bar{F}(n^{-2}C_n) = (-\beta+o_p(1))\log (n^{-2}C_n),$ on $(W>0)$.  Combining the
  above mentioned asymptotics we get, on $W>0$,
  \begin{align*}
    \log \left(p_n \sum_{k=1}^{n-1}Z_k \right)
    & = \log p_n +  \log \sum_{k=1}^{n-1}Z_k =-(\beta+o_p(1))\log C_n +2(\beta+o_p(1)) \log n + \log \sum_{k=1}^{n-1} Z_k \\
    & =-(1+o_p(1))\log Z_n +2(\beta+o_p(1)) \log n + \log \sum_{k=1}^{n-1} Z_k \\
    & = -(1+o_p(1))\log Z_n +2(\beta+o_p(1)) \log n + (\alpha+o_p(1)) \log Z_n \to - \infty \;  
  \end{align*}
  as $\log Z_n$ grows exponentially on $(W>0)$ and $\alpha<1$. In the last line, we have used
  \cref{conv1} with $s=1$.  This concludes the proof.
\end{proof}

Some corollaries can easily
be derived from \cref{th:regvarmain} about the asymptotic behaviour of the ordered statistic of the
displacement of the particles in the $n$-th generation as $n \to \infty$.

\begin{corollary}
  \label{cor1}
  Fix $k \geq 1$. Let, $M_n^{(k)}$ be the $k$-th maximum value in the set
  $\left\{S_v : v \in D_n \right\}$. Then under assumptions of \cref{th:regvarmain}, we have
  \begin{align}{\label{probconv:reg}}
    \mathbb{P}(C_n^{-1}M_n^{(k)} \leq x) \to \exp(-x^{-\beta})\Big[\sum_{j=0}^{k-1} \frac{1}{j!}
    x^{-\beta j} \Big]
  \end{align}
  as $n \to \infty$ conditioned on the survival of the tree. Moreover,
  \begin{align}{\label{reg:inprob}}
    \alpha^n \log M_n^{(k)} \stackrel{P}{\longrightarrow} \frac{W}{\beta}, \;\; \text{conditioned on
    survival}.
  \end{align}
\end{corollary}

\begin{proof}
  For any $x>0$ and $k\ge 1$ we have from \cref{th:regvarmain}
  \begin{align*}
  \mathbb{P}(N_n ((x, \infty]) \leq k-1 | \mathbb{T} \; \text{survives}) \to
  \mathbb{P}(\mathrm{Poisson}(x^{-\beta})\leq k-1).
  \end{align*}
  Also, observe that $(N_n ((x, \infty]) \leq k-1)=(C_n^{-1}M_n^{(k)} \leq x)$ and
  \begin{align*}
  \mathbb{P}(\mathrm{Poi}(x^{-\beta})\leq k-1) =\exp(-x^{-\beta})\Big[\sum_{j=0}^{k-1} \frac{1}{j!}
  x^{-\beta j} \Big] = H(x)
  \end{align*}
  for some distribution function $H$ which gives mass on $(0,\infty)$. Therefore,
  \begin{equation}
    \label{eq:k-max}
    C_n^{-1}M_n^{(k)} \stackrel{D}{\longrightarrow} H, \;\;
    \text{conditioned on survival }.
  \end{equation}
  This proves the first part of the result. On taking logarithms on both sides of \eqref{eq:k-max}
  and multiplying by $\alpha^n$ we have
  \begin{equation}
    \label{correq}
    \alpha^n\log M_n^{(k)} -\alpha^n \log C_n \stackrel{P}{\longrightarrow} 0.
  \end{equation}
  Under Assumption~\ref{item:R3}, $(1/(1-F))^{\leftarrow} \in RV_{1/\beta}$. Using \cite[Proposition 0.8]{Resnick}, we have,
  \begin{align*}
    \frac{\alpha^n \log C_n}{\alpha^n \log Z_n} = \frac{\log C_n}{\log Z_n} = \frac{\log \left[(1/(1-F))^{\leftarrow} (Z_n)\right]}{\log Z_n}
    \stackrel{a.s.}{\longrightarrow} \frac{1}{\beta}, \;\; \text{conditioned on survival}.
  \end{align*}
  Finally using \cref{davies} we have
  \begin{align*}
    \alpha^n \log C_n \stackrel{a.s.}{\longrightarrow} \beta^{-1} W,
  \end{align*}
  conditioned on survival. Using~\eqref{correq} we have the desired result.
\end{proof}

\Cref{cor1} says that the rightmost particles of the BRW go away from the origin in a
double-exponential speed in this set-up. Such double-exponential growth was also observed only for
the rightmost particle in a related setup by \cite{van2017explosion}. Indeed we can improve the convergence in (\ref{probconv:reg}) to almost sure convergence as demonstrated by \cref{aslim:regtail}.
 \begin{theorem}
   \label{aslim:regtail}
  For any $k \in \mathbb{N}$, under the assumptions of \cref{th:regvarmain}, conditionally on the survival of the Galton-Watson
   tree $\mathbb{T}$,
   \begin{align*}
     \lim_{n \to \infty} \alpha^n  \log  M_n^{(k)} = \beta^{-1}W,  \; \text{almost surely}.
   \end{align*}
 \end{theorem}
 \begin{proof}[Proof of \cref{aslim:regtail}]
 We shall first prove the lower bound. Non-negativity of the displacement variables imply that $ M_n^{(k)}$ is greater than the  $k$-th maximum of the collection $\left\{X_{e_v} : v \in D_n \right\}$.  Fix any $\varepsilon >0$ and $\beta_1 > \beta$. On the event $(W > 2\varepsilon)$, We have the following for large enough $n$.
 \begin{align}
\log  \mathbb{P} \left( \alpha^n \log M_n^{(k)} \leq \beta_1^{-1}(W-2\varepsilon) \Big \rvert \mathbb{T} \right) & \leq \log  \mathbb{P} \left( \alpha^n \log \left( k\text{-th maximum of } \left\{X_{e_v} : v \in D_n \right\} \right) \leq \beta_1^{-1}(W-2\varepsilon) \Big \rvert \mathbb{T} \right) \nonumber  \\ 
& \leq  \log \left\{ {Z_n \choose Z_n-k+1 }\left[ F \left( \exp\left( \beta_1^{-1}(W-2\varepsilon) \alpha^{-n}\right)\right)\right]^{Z_n-k+1} \right\}  \nonumber \\
& \leq k \log Z_n + (Z_n-k+1) \log \left[ 1- \bar{F}\left( \exp\left( \beta_1^{-1}(W-2\varepsilon) \alpha^{-n}\right)\right) \right] \nonumber \\
& \leq k \log Z_n - (Z_n-k+1) \bar{F}\left( \exp\left( \beta_1^{-1}(W-2\varepsilon) \alpha^{-n}\right)\right) \nonumber  \\
& \leq k \log Z_n - C_1(Z_n-k+1)\exp \left( -(W-2\varepsilon)\alpha^{-n} \right) \label{aslw},
 \end{align}
 for some (non-random) constant $C_1 >0$. We have used Potter's bound \cite[Proposition~0.8(ii)]{Resnick} in (\ref{aslw}).  By \cref{davies}, we have, conditionally on survival of the tree, 
 $$ \alpha^n  \log \left[C_1(Z_n-k+1)\exp \left( -(W-2\varepsilon)\alpha^{-n} \right) - k \log Z_n \right] \stackrel{a.s.}{\longrightarrow} 2 \varepsilon,$$
 and hence
 $$  \mathbb{P} \left( \alpha^n \log M_n^{(k)} \leq \beta_1^{-1}(W-2\varepsilon) \Big \rvert \mathbb{T} \right) \leq \exp \left( - \exp \left( \varepsilon \alpha^{-n} \right) \right), $$
 for all $n \geq N_1(\varepsilon)$, where $N_1(\varepsilon)$ is finite almost surely. Applying First Borel-Cantelli Lemma and taking $\beta_1 \downarrow \beta$, we can conclude 
   $\liminf_{n \to \infty} \alpha^n \log M_n^{(k)} \geq \beta^{-1}(W-2\varepsilon)$ almost surely on the event $(W > 2 \varepsilon).$ Recalling that the event $(W>0)$ and the event of survival of the tree are same almost surely, we take $\varepsilon \downarrow 0$ and complete the proof of the lower bound.
  
It is enough to prove the upper bound for $k=1$. Observe that for fixed
   $\beta_2 \in (0,\beta)$ and $\varepsilon>0$, by Potter's bound \cite[Proposition~0.8(ii)]{Resnick}, we
   have for some constant $C_2>0$ and $n$ large enough,
   \begin{align*}
     \mathbb{P}\left(\alpha^n \log M_n^{(1)} \geq \beta_2^{-1}(W + 2\varepsilon) \Big \rvert  \mathbb{T}\right) 
     & \leq  \sum_{v \in D_n}  \mathbb{P}\left(S_v \geq \exp\left(  \beta_2^{-1}(W + 2\varepsilon) \alpha^{-n}\right) \Big \rvert \mathbb{T}\right)\\
     & \leq  Z_n n \overline{F}\left(\frac{1}{n} \exp\left(  \beta_2^{-1}(W + 2\varepsilon) \alpha^{-n}\right)  \right) \\
     & \leq  C_2 Z_n n^{\beta_2+1} \exp\left(-\alpha^{-n} (W+2\varepsilon)\right) \; .
   \end{align*}
    By \cref{davies}, we have, conditionally on survival of the tree, 
    $$ \alpha^n  \log \left[C_2 Z_n n^{\beta_2+1} \exp\left(-\alpha^{-n} (W+2\varepsilon)\right) \right] \stackrel{a.s.}{\longrightarrow} - 2 \varepsilon,$$
    and hence
    $$  \mathbb{P} \left( \alpha^n \log M_n^{(1)} \geq \beta_2^{-1}(W+2\varepsilon) \Big \rvert \mathbb{T} \right) \leq \exp \left( - \varepsilon \alpha^{-n}  \right), $$
    for all $n \geq N_2(\varepsilon)$, where $N_2(\varepsilon)$ is finite almost surely. Applying First Borel-Cantelli Lemma and taking $\beta_2 \uparrow \beta$, we  conclude  the proof.
 \end{proof}
 

\section{BRW with rapidly varying tails }
\label{section:ltd}


\begin{assumption}
  \label{assumpF}
  $F$ is a distribution function on $\mathbb{R}$ with the following properties.
  \begin{enumerate}[leftmargin=*,label=(F\arabic*)]
  \item \label{item:F1} $F(x)<1, \;\forall \; x \in \mathbb{R},$ and $K:=-\log(1-F)$ is regularly
    varying at $\infty$ with index $r \in (0,\infty)$.
  \item \label{item:F2} For some $M \in (0, \infty)$ and $\eta>0$,
    $F(-t) \leq M(\log |t|)^{-2-\eta}$, for $t$ large enough.
  \end{enumerate}
\end{assumption}
Examples of distributions satisfying \ref{item:F1} include the Weibull distributions with
$K(x) = cx^{r}$ and the Gaussian distribution with $K(x) \sim x^2/2$. Assumption \ref{item:F1} is
very close to implying that $F$ is in the domain of attraction for the maximum of the Gumbel
distribution, but some additional condition on $K$ is needed to ensure that property; see
e.g. \cite{cline:1986} for examples and counterexamples.

Let $L$ be the left-continuous inverse of $K$, that is
\begin{align}{\label{defL}}
  L(u) = \inf\{x:K(x)>u\} \; ,  \ \ u\in(0,1) \; .
\end{align}
Then $F^{\leftarrow} (x) = L(-\log(1-x))$, for all $0 <x <1$ and $L$ is regularly varying at
$\infty$ with index $1/r$; cf.  \cite[Proposition 0.8(v)]{Resnick}. By the regular variation of $L$ and
\cref{davies}, it is not hard to see the following formula (cf.~\cite[Proposition~0.8]{Resnick}), that for all $a>0$,
  \begin{align}
  \label{eq:argument-regvar-Z}
  \lim_{n\to\infty} \frac{L(a\log Z_n)}{L(\alpha^{-n})} & = a^{1/r}W^{1/r} \; , \ \ \textrm{almost surely conditioned on survival} \; .
\end{align}
  We now state our main result.

\begin{theorem}
  \label{lighttailmain}
  Let \cref{davies_assump,assumpF} hold. Then, almost surely conditioned on survival,
  \begin{align*}
    \lim_{n\to\infty} \frac{M_n}{L(\log Z_n)}
  & = \begin{cases} (1-\alpha^{\frac1{r-1}})^{\frac1r-1}, &\mbox{if } r > 1, \\
    1, & \mbox{if } 0<r \leq1. \end{cases}, 
  \end{align*}
\end{theorem}
As an immediate corollary of \cref{lighttailmain} and \cref{eq:argument-regvar-Z}, we obtain
that almost surely conditioned on survival,
\begin{align}
  \label{eq:deterministic}
  \lim_{n\to\infty} \frac{M_n}{L(\alpha^{-n})} & = [(1-\alpha^{\frac1{r-1}})_{+}^{\frac1r-1} \vee 1]W^{\frac1r} \; ,
\end{align}
where $x_+$ denotes the non-negative part of $x$.

\paragraph{\bf Comments and examples} The case $r<1$ contains subexponential distributions in the
domain of attraction (for the maximum) of the Gumbel law. The heuristic properties of such
distributions is the ``single large jump principle'', which means that the sum of a finite number of
i.i.d. random variables having a subexponential distribution is large when exactly one of the
summands is large. Given this principle and the double exponential growths of the population, it is
expected that the maximum $M_n$ will be large only if one displacement of the last generation is
large. Thus \cref{lighttailmain} confirms the intuition in the case $r<1$. This is similar to
  the case of regularly varying displacements considered in \cref{th:regvarmain}. However, we could
  not obtain the convergence of the point process of exceedance in the present case, and we do not
  know if the same result (as in the case of regularly varying displacements) can be expected. The
case $r=1$ contains distributions which are subexponential (take for instance
$K(x)=x(\log x)^{-\beta}$ with $\beta>0$,
 \cite[cf. Example~1.4.3~(b)]{embrechts:kluppleberg:mikosch:1997}), and distributions which are not
subexponential such as the exponential distribution. However, we obtain the same result as in the
subexponential case.

The case $r>1$ is more intriguing. It contains distributions which are also in the domain of
attraction of the Gumbel law, but are not subexponential. Heuristically, this implies that a sum of
i.i.d.~random variables with such a distribution is large when many or all terms contribute
significantly to the sum. However, the double exponential growth of the population implies that the
displacements of the older generation contribute less. The form of the limit hints at exponential
smoothing. More precisely, if $K(x) \sim c x^r$, it can be shown, using results on the tail of the
sum of i.i.d.~random variables with such distribution, see
e.g. \cite[Theorem~1.1]{balkema:kluppelberg:resnick:1993} or \cite[Theorem~4.1]{asmussen:hashorva:laub:taimre:2017}, that the approximate
asymptotic behaviour of $M_n$ is the same as the behaviour of
\begin{align*}
  \max_{1 \leq j \leq Z_n}\sum_{i=1}^n \alpha^{\frac{n-i}{r}} X_i^{(j)}
\end{align*}
where $X_i^{(j)}$ are i.i.d.~and independent of $Z_n$.
For Gaussian displacements which corresponds to $K(x) \sim x^2/2$, we have 
$L(x)\sim \sqrt{2x}$ and \cref{lighttailmain} reads
\begin{align*}
  \lim_{n\to\infty} \frac{M_n}{\sqrt{\log Z_n}} = \sqrt{\frac{2}{1-\alpha}} \; , \ \ \mbox{ almost surely conditionally on survival,}
\end{align*}
and \cref{eq:deterministic} becomes
\begin{align*}
  \lim_{n\to\infty} \alpha^{n/2}M_n = \sqrt{\frac{2W}{1-\alpha}} \; , \ \ \mbox{ almost surely conditionally on survival.}
\end{align*}


We now turn to the proof of \cref{lighttailmain}, for which we will need the following lemma whose
proof is in \Cref{app:prooflems}.
\begin{lemma}
  \label{normordernew}
  Suppose that we have a filtration $\left\{\mathcal{F}_n : n \geq 0\right\}$ on the probability space $(\Omega, \mathcal{F}, \mathbb{P})$ and an $\mathcal{F}$-adapted non-negative integer-valued process $\left\{\zeta_n : n \geq  0\right\}$ such that $\liminf_{n \to \infty} n^{-1}\log \zeta_n >0$, almost surely. Further assume that for all $n \geq 1$, there exists a collection of random variables, $\left\{G_{n,i} : i \leq \zeta_{n}\right\}$ satisfying the following two conditions.
  \begin{enumerate}
  \item The collection $\left\{G_{n,i} : i \leq \zeta_{n}\right\}$ is $\mathcal{F}_{n+1}$-measurable.
  \item Conditioned on $\mathcal{F}_{n}$, $\left\{G_{n,i} : i \leq \zeta_{n}\right\}$ is a collection of i.i.d. variables with distribution
    function $F$ satisfying 
    $F(x)<1$ for all $x$ and $K(L(x)) \sim x$ as $x \to \infty$, where $K=-\log (1-F)$ and $L$ is the left-continuous inverse of $K$. Also assume that 
    $$ \lim_{t \to 1} \lim_{x \to \infty} \dfrac{L(tx)}{L(x)} =1.$$
  \end{enumerate}
   Fix  a sequence $\left\{\psi_n\right\}_{n \geq 1}$ in $(0,1)$, bounded away from $1$ and let $l_n$ be another $\mathcal{F}$-adapted non-negative integer valued process such that $\log l_n \stackrel{a.s.}{\sim} (1-\psi_n) \log \zeta_n$ and $\psi_n \log \zeta_n \stackrel{a.s.}{\rightarrow} \infty$, as $n \to \infty$.  Let $G_{i:\zeta_{n}}$ be the $i$-th largest element in
     $\left\{G_{n,1},\dots,G_{n,\zeta_{n}}\right\}$. Then,
     \begin{align*}
        \frac{G_{l_{n}:\zeta_{n}}}{L(\psi_n \log \zeta_{n})} \stackrel{a.s.}{\longrightarrow} 1, \; \text{ as } n \to \infty.\;
     \end{align*}
    The assertion is also true if we take $l_n=\psi_n=1$ for all $n \geq 1$.
  \end{lemma}
  
 \begin{remark}{\label{Fsat}}
If $F$ satisfies \cref{item:F1} then the assumptions on $F$, stated in \cref{normordernew}, holds true. By~\cite[Theorem 1.5.12]{bingham}, we have $K(L(x)) \sim x$ as $x \to \infty$, whereas $L(tx)/L(x) \to t^{1/r}$ as $x \to \infty$, since $L$ is regularly varying with index $1/r$.
\end{remark}

For i.i.d.~regularly varying random variables (i.e. when $G_n=G$ for all $n$), the next result is a stronger version of the law of iterated logarithm; see e.g. \cite{wichura:1972}. Since \cref{assump:inhom} does not imply regular
 variation and we deal with triangular arrays with random row sizes, we provide a proof  based on estimates
 of the Laplace transform of the distributions $G_n$ due to \cite{Davies} in Section \ref{app:prooflems}.
 \begin{lemma}
   \label{strongheavy1}
 Fix a sequence of distribution functions $\left\{G_n : n \geq 0\right\}$ satisfying
   \cref{assump:inhom} with moment index sequence $\left\{\alpha_n : n \geq 0\right\}$. Suppose that we have a filtration $\left\{\mathcal{F}_n : n \geq 0\right\}$ on the probability space $(\Omega, \mathcal{F}, \mathbb{P})$ and an $\mathcal{F}$-adapted positive integer-valued process $\left\{\zeta_n : n \geq  0\right\}$ such that $\liminf_{n \to \infty} n^{-1}\log \log \zeta_n >0$, almost surely.  Assume that for all $n \geq 0$, there exists a collection of random variables, $\left\{L_{n,i} : i \leq \zeta_{n}\right\}$ satisfying the following two conditions.
 \begin{enumerate}
 \item The collection $\left\{L_{n,i} : i \leq \zeta_{n}\right\}$ is $\mathcal{F}_{n+1}$-measurable.
 \item Conditioned on $\mathcal{F}_{n}$,  $\left\{L_{n,i} : i \leq \zeta_{n}\right\}$ is a collection of i.i.d. variables having distribution function $G_{n}$. 
 \end{enumerate}
  Then we can find a positive integer valued finite random variable $N$ such that 
  $$ 0 < \prod_{n \geq N+1}  \dfrac{\alpha_{n} \log \sum_{i=1}^{\zeta_{n}} L_{n,i} }{\log \zeta_{n}} < \infty, \; \text{ almost surely }.$$
  Consequently, \begin{align*}
     \frac{\alpha_{n} \log \sum_{i=1}^{\zeta_{n}} L_{n,i} }{\log \zeta_{n}} \stackrel{a.s.}{\longrightarrow} 1, \; \text{ as } n \to \infty.
   \end{align*}
 \end{lemma}

We now have all the necessary ingredients for the proof of \cref{lighttailmain}. We shall split the
proof into a lower and an upper bound. 

\begin{proof}[Proof of \cref{lighttailmain}, lower bound]
  The idea to obtain a lower bound is to truncate the tree at the $k$-th generation from the
    last one.  We obtain a bound for the displacements up to the generation $n-k$ of the form
    $\vartheta_k L(\log Z_n)$ with $\vartheta_k\to0$ as $k \to\infty$. The most delicate part of the
    proof is then to choose  subsets of vertices from each of the remaining generations such that all the vertices in those subsets are among the ones having largest displacements from their parents in their respective generations but still having  a path from the $(n-k)$-th generation through these subsets to the last generation. The first criterion forces us to choose the subsets as small as possible while the second one requires the subsets to be large. As a trade-off between them we get our required lower bound.

    
    To be
    more precise, fix an integer $k \geq 1$ and $\delta_0, \ldots, \delta_k \in (0,1)$. We define
  recursively a finite sequence of subsets of vertices from the tree. Set $B_{n,n-k} = D_{n-k}$. Let
  $A_{n,n-k}$ be the set of vertices in $D_{n-k}$ whose displacements from their parents are among
  the largest $\lfloor Z_{n-k}^{1-\delta_k} \rfloor $ many of those in $B_{n,n-k}$. For
  $i=1,\ldots,k$; let $B_{n,n-k+i}$ be the set of children of the vertices in $A_{n,n-k+i-1}$ and let
  $A_{n,n-k+i}$ be the set of vertices in $B_{n,n-k+i}$ whose displacements from their parents are
  among the largest $ \lfloor |B_{n,n-k+i}|^{1-\delta_{k-i}} \rfloor$ many of those in
  $B_{n,n-k+i}$. We always break the ties uniformly at random whenever some of the displacements are equal.

Our construction has been done in such a way that on the event $(|A_{n,n}| > 0)$, there exists at least one ray of
  vertices starting from $A_{n,n-k}$ and going through $A_{n,n-k+i}$ for all $0 \leq i \leq k$. Let
  $v^{*}_{n-k}$ be one such vertex in $A_{n,n-k}$ generating such a ray. Then the following crucial observation will be pivotal in proving our lower bound.
\begin{align}{\label{lowerboundcruc}}
  M_n \mathbbm{1}\left( |A_{n,n}| > 0\right)\geq \mathbbm{1}\left( |A_{n,n}| > 0\right)\max_{v \in A_{n,n}} S_v \geq \mathbbm{1}\left( |A_{n,n}| > 0\right) \left[S_{p(v_{n-k}^\ast)} +  \sum_{i=0}^k \min_{v \in A_{n,n-k+i} } X_{e_v} \right].
  \end{align}
The lower bound will then follow after establishing appropriate lower bounds for the two terms within bracket in the right hand side of \cref{lowerboundcruc}. All the almost sure statements, which appear in the following proof, are assumed to be conditioned on the survival of the tree.

We start by obtaining a  lower bound to the first term in the right hand side of \cref{lowerboundcruc}. Notice that if we condition on $\left\{Z_i : i \geq 1 \right\}$, independence of progeny and
  displacement variables guarantee that $p(v^{*}_{n-k})$, the parent of the vertex $v^{*}_{n-k}$, is
  a uniform random element from $D_{n-k-1}$ after conditioning and hence $S_{p(v^{*}_{n-k})}$ has distribution
  $F^{*(n-k-1)}$, where $F^{*l}$ denotes the $l$-fold convolution of $F$. Fix an arbitrary $\varepsilon >0$ and observe that for large enough $n$,
  \begin{align}
    \mathbb P \left( L(\alpha^{-n})^{-1} S_{p(v^\ast_{n-k})}  \le -\varepsilon \Big \vert \left\{Z_i \right\}_{i \geq 1}\right)
    & = F^{*(n-k-1)}(-\varepsilon L(\alpha^{-n})) \nonumber\\
    & \leq (n-k-1)F(-\varepsilon L(\alpha^{-n})/(n-k-1)) \label{union}\\
    & \leq nF(-\varepsilon L(\alpha^{-n})/n) \nonumber\\
    & \leq   Mn (\log \varepsilon + \log L(\alpha^{-n}) -\log  n)^{-2-\eta}  \leq M_1n^{-1-\eta}, \label{ineq2}
  \end{align} 
  for some finite positive constant $M_1$, depending on $\varepsilon, \alpha,r$ and $M$. The inequality~\cref{union} follows from a simple application of union bound, whereas the inequalities in~\cref{ineq2} follow from Assumption~\ref{item:F2} and the fact that $\log L(\alpha^{-n}) \sim -\frac{n}{r}\log\alpha$; see
  \cite[Proposition~0.8(i)]{Resnick}.  Since the final expression is summable and $\varepsilon$ is
  arbitrary, an application of the first Borel-Cantelli Lemma implies that $\liminf_{n \to \infty} L(\alpha^{-n})^{-1} S_{p(v_{n-k}^\ast)} \geq 0$, almost
  surely. 

As a first step to get a lower bound on the second term in the right hand side of (\ref{lowerboundcruc}), we make the following two claims.
\begin{enumerate}
\item $\liminf_{n \to \infty} \alpha^n \log |A_{n,n-k+i}|, \liminf_{n \to \infty} \alpha^n \log |B_{n,n-k+i}| >0$, almost surely, for all $0 \leq i \leq k$.
\medskip
\item  $\log |A_{n,n-k+i}| \stackrel{a.s.}{\sim} \alpha^{-1}(1-\delta_{k-i}) \log |A_{n,n-k+i-1}|$, for all $1 \leq i \leq k$.
\end{enumerate}
The first claim, in particular, implies that $|A_{n,n}| \stackrel{a.s.}{\longrightarrow} \infty$, leading to the following simplification of \cref{lowerboundcruc}.
\begin{equation}{\label{lowerboundcruc2}}
 \liminf_{n \to \infty} L(\alpha^{-n})^{-1} M_n =\liminf_{n \to \infty} L(\alpha^{-n})^{-1} \left[S_{p(v_{n-k}^\ast)} +  \sum_{i=0}^k \min_{v \in A_{n,n-k+i} } X_{e_v} \right].
\end{equation}
The proof of the claims will be carried out by an induction on $i$. Note that, 
 \begin{equation}{\label{induc1}}
 \alpha^n \log |A_{n,n-k}|  \stackrel{a.s.}{\sim} \alpha^n (1-\delta_k) \log |B_{n,n-k}| = \alpha^n (1-\delta_k) \log Z_{n-k} \stackrel{a.s.}{\sim} \alpha^k (1-\delta_k)W >0.
 \end{equation} 
 This establishes the first claim for $i=0$. The vertices
  in $A_{n,n-k}$ produce i.i.d.~numbers off-springs (independent of their current positions) having
  distribution $G$ and their sum (i.e., the total number of children) is $|B_{n,n-k+1}|$; thus \cref{strongheavy1} gives
  $\log |B_{n,n-k+1}| \stackrel{a.s.}{\sim} (1/\alpha) \log |A_{n,n-k}|.$ Here we have applied \cref{strongheavy1} with $\mathcal{F}_{n}$ being the $\sigma$-algebra which contains all the information (tree structure and displacement) of the tree upto generation $(n-k)$, $\zeta_n = |A_{n,n-k}|$ and $G_n \equiv G$. The validity of the two conditions needed to apply \cref{strongheavy1} is guaranteed by (\ref{induc1}) and the independence of displacements in one generation of the tree from the information of previous generations. Finally, the identity
  $|A_{n,n-k+1}| = \lfloor |B_{n,n-k+1}|^{1-\delta_{k-1}} \rfloor $ yields that
  $\log |A_{n,n-k+1}| \stackrel{a.s.}{\sim} \alpha^{-1}(1-\delta_{k-1}) \log |A_{n,n-k}|$. This proves the induction
  hypothesis for the base case $i=1$. The proof of the veracity of the induction hypothesis for
  $i=l+1$ assuming it is true for $i=l$ goes exactly along the same lines. 



  
  Using the claim established just above, we can conclude that
  \begin{align}
    \label{eq:logAnk}
    \alpha^n \log |A_{n,n-k+i}| \to \alpha^{k-i} W \prod_{j=0}^i (1-\delta_{k-j})  \; , 
  \end{align}
  almost surely, 
   for all $0 \leq i \leq k$.
In particular, $|A_{n,n-k+i}|$ and hence $|B_{n,n-k+i}|$ grows doubly-exponentially almost surely.

  Now observe that by construction $\min_{v \in A_{n,n-k+i} } X_{e_v}$ is the
  $\lfloor |B_{n,n-k+i}|^{1-\delta_{k-i}} \rfloor$-th largest element in the collection
  $\left\{ X_{e_v} : v \in B_{n,n-k+i}\right\}$. Also notice that conditioned on the information of the tree upto generation $(n-k+i)$, the collection  $\left\{ X_{e_v} : v \in B_{n,n-k+i}\right\}$ is just an i.i.d.~sample from distribution function $F$. Hence, using \cref{normordernew} and the fact
  that $|B_{n,n-k+i}| $ grows doubly-exponentially almost surely, we obtain the following for $0 \leq i \leq k$.
  \begin{align*}
    \liminf_{n \longrightarrow \infty} \frac{\min_{v \in A_{n,n-k+i}  } X_{e_v}}{L( \delta_{k-i} \log |B_{n,n-k+i}|)} \geq 1 \; , 
    \ \ \text{almost surely}. 
  \end{align*}
  Here we have applied \cref{normordernew} with $\mathcal{F}_{n}$ being the $\sigma$-algebra which contains all the information (tree structure and displacement) of the tree upto generation $(n-k+i)$, $\zeta_n = |B_{n,n-k+i}|$ and $\psi_n \equiv \delta_{k-i}.$
  Again using the fact $|A_{n,n-k+i}| = \lfloor |B_{n,n-k+i}|^{1-\delta_{k-i}} \rfloor$ and
  regularly variation of $L$, we obtain for $0 \leq i \leq k$ the following.
  \begin{align}
    \label{eq:liminfprod}
    \liminf_{n \longrightarrow \infty} \frac{\min_{v \in A_{n,n-k+i} } X_{e_v}}{L( \log
    |A_{n,n-k+i}|)} \geq \left(\frac{\delta_{k-i}}{1-\delta_{k-i}}\right)^{1/r}, \;\text{almost
    surely}. 
  \end{align}
  
  By \cref{eq:logAnk}, \cref{eq:davies:conv} and the regular variation of $L$ and the Uniform
  Convergence Theorem for regularly varying functions, we have
  \begin{align}
    \label{eq:limregvar}
    \lim_{n\to\infty} \frac{ L(\log |A_{n,n-k+i}|)}{L(\alpha^{-n})} = \left(\alpha^{k-i} \prod_{j=0}^i (1-\delta_{k-j}) W\right)^{1/r}.
  \end{align}
  The limits \cref{eq:liminfprod,eq:limregvar} yield, for $i=0,\dots,k$, 
  \begin{align*}
    \liminf_{n \to \infty} \frac{ \min_{v \in A_{n,n-k+i} } X_{e_v}}{L(\alpha^{-n})} \geq
    \left( \alpha^{k-i} \frac{\delta_{k-i}}{1-\delta_{k-i}}\prod_{j=0}^i (1-\delta_{k-j}) W  \right)^{1/r} \; .
  \end{align*}
  Combining all the previous limit estimates for the terms in right hand side of (\ref{lowerboundcruc2}) and
  using \cref{eq:argument-regvar-Z},
  we obtain
  \begin{align*}
    \liminf_{n \longrightarrow \infty} \frac{M_n}{L(\log Z_n)} \geq \sum_{i=0}^k \left( \alpha^{k-i}
    \frac{\delta_{k-i}}{1-\delta_{k-i}}\prod_{j=0}^i (1-\delta_{k-j}) \right)^{1/r}.
  \end{align*}
  Maximizing over the right hand side we conclude that
  \begin{align*}
  \liminf_{n \longrightarrow \infty} \frac{M_n}{L(\log Z_n)} \geq f_k := \sup \left\{ \sum_{i=0}^k
    \left( \alpha^{k-i} \frac{\delta_{k-i}}{1-\delta_{k-i}}\prod_{j=0}^i (1-\delta_{k-j})
    \right)^{1/r} \; \Big \vert \; \delta_0, \ldots, \delta_k \in (0,1) \right\}.
  \end{align*}
  In order to compute $f_k$, we define, for $k\geq1$ and
  $\delta_0, \ldots, \delta_{k} \in (0,1)$,
  \begin{align*}
    h_{k}(\delta_0, \ldots, \delta_{k-1})
    & := \sum_{i=0}^k \left( \alpha^{k-i}  \frac{\delta_{k-i}}{1-\delta_{k-i}}\prod_{j=0}^i (1-\delta_{k-j}) \right)^{1/r} 
      = \sum_{i=0}^k \left( \alpha^{i} \delta_{i}\prod_{q=i+1}^{k} (1-\delta_{q}) \right)^{1/r} \; , \\
    f_k & = \sup \left\{h_k(\delta_0, \ldots, \delta_k) : \delta_0, \ldots, \delta_k \in (0,1) \right\}
  \end{align*}
  Taking $\delta_0 \uparrow 1$ and $\delta_i \downarrow 0$ for $i\geq1$ yields $f_k\geq1$.  Furthermore,
  $$h_k(\delta_0, \ldots, \delta_k) = \left(\alpha^k \delta_k\right)^{1/r} + {(1-\delta_k)}^{1/r}
  h_{k-1}(\delta_0, \ldots, \delta_{k-1}),$$ and hence we have for all $k \geq 1$,
  \begin{align*}
    f_k 
    &= \sup \left\{\alpha^{k/r}\delta_k^{1/r} + ({1-\delta_k})^{1/r} h_{k-1}(\delta_0, \ldots, \delta_{k-1}) :
      \delta_0, \ldots, \delta_k \in (0,1) \right\} \\
    &=\sup \left\{\alpha^{k/r}\delta_k^{1/r}  + ({1-\delta_k})^{1/r} f_{k-1} :  \delta_k \in (0,1) \right\} \\
    & =\begin{cases} \left(\alpha^{k/(r-1)}+f_{k-1}^{r/(r-1)}\right)^{1-1/r}, &\mbox{if } r > 1, \\
      \max \left\{ \alpha^{k/r}, f_{k-1}\right\}, & \mbox{if } 0<r \leq 1. \end{cases}  
  \end{align*}  
  For $r>1$, opening the recursion we get
  $f_{k}^{r/(r-1)} = f_0^{r/(r-1)} + \sum_{i=1}^k \alpha^{i/(r-1)}.$ Having observed that
  $ f_0 = \sup \left\{ \delta_0^{1/r} : \delta_0 \in (0,1)\right\} =1$, we conclude
  \begin{align*}
    f_k^{r/(r-1)} = \begin{cases} \sum_{i=0}^k \alpha^{i/(r-1)}, &\mbox{if } r > 1, \\  1, & \mbox{if } 0<r \leq 1. \end{cases}.
  \end{align*}
Letting $k$ tend to $\infty$, we therefore conclude that almost surely,
  \begin{align*}
    \liminf_{n \longrightarrow \infty} \frac{M_n}{L(\log Z_n)} & \geq \begin{cases} (1-\alpha^{1/(r-1)})^{\frac1r-1}, &\mbox{if } r > 1, \\
      1, & \mbox{if } 0<r \leq 1. \end{cases} \\
          &  =  (1-\alpha^{1/(r-1)})_{+}^{\frac1r-1}\vee1 \; .
  \end{align*}
  This proves the lower bound. 
\end{proof}
We now proceed to prove the upper bound. We first establish a preliminary non sharp upper bound
which will be used to obtain the sharp one. This preliminary bound is basically obtained by summing over maximum displacement from the parent in each generation.
\begin{lemma}
  Under the assumptions of \cref{lighttailmain}, there exists an almost surely finite random
  variable~$A$ such that
  \begin{equation}
    \label{prelimupper}
    \limsup_{n \to \infty} \frac{M_n}{L(\log Z_n)} \leq  \frac{AW^{1/r}}{1-\alpha^{1/r}}, \;\text{almost surely}.
  \end{equation}
\end{lemma}
\begin{proof}
  Introduce the notation
  \begin{align*}
    \widetilde M_{Z_i}:= \max_{v\in D_i} X_{e_v}.
  \end{align*}
  Note that $M_n\le \sum_{i=1}^n \widetilde M_{Z_i}$. Since $K(L(x)) \geq x$, for all $x$, we get
  \begin{align*}
    \mathbb{P} \left[ \widetilde{M}_{Z_i} > L((1+\varepsilon)\log Z_i)\Big \vert \mathbb{T}\right]
    & = 1- \left(F \circ L((1+\varepsilon)\log Z_i)\right)^{Z_i} \\
    &= 1- \left(1- \exp(-K\circ L((1+\varepsilon)\log Z_i))\right)^{Z_i} \\
    & \leq 1- \left(1- \exp(-(1+\varepsilon)\log Z_i)\right)^{Z_i} \\
    & = 1- \left(1-  Z_i^{-1-\varepsilon}\right)^{Z_i} \leq Z_i^{-\varepsilon} \; .
  \end{align*}
  Since the last expression is summable, we have by Borel-Cantelli Lemma, 
  almost surely conditioned on survival,
 \begin{equation}\label{eq:upperbdM-tildeM1}
     M_n\le \sum_{i=1}^n \widetilde M_{Z_i} \le \sum_{i=1}^n L((1+\varepsilon)\log Z_i)+ \sum_{i=1}^{\infty} \left( \widetilde{M}_{Z_i} \vee 0 \right) \mathbbm{1} \left(  \widetilde{M}_{Z_i} > L((1+\varepsilon)\log Z_i)\right), 
  \end{equation}
  where the second expression in the right hand side of (\ref{eq:upperbdM-tildeM1}) is finite almost surely. 
  By \cref{eq:argument-regvar-Z},
  there exists a random finite number $A$ such that
  $L((1+\varepsilon)\log Z_i) \leq AL(\alpha^{-i})$ for all $i\geq1$, almost surely. Therefore,
  almost surely conditionally on survival,
  \begin{align*}
    \limsup_{n \to \infty} \frac{M_n}{L(\alpha^{-n})}
    & \leq  \limsup_{n \to \infty} \frac1{L(\alpha^{-n})} \sum_{i=1}^n  L((1+\varepsilon)\log Z_i)  \\
    & \le A \limsup_{n \to \infty} \frac1{L(\alpha^{-n})} \sum_{i=1}^{ n} L(\alpha^{-i})  =   \frac{A}{1-\alpha^{1/r}} \; ,
  \end{align*}
  where the last assertion follows from \cref{regulargeom}.  By \cref{davies} and
  \cref{eq:argument-regvar-Z}, this proves \cref{prelimupper}.
\end{proof}

Before diving into the main part of the proof for the upper bound, let us try to explain in a few words how does the proof work. We first fix some $k \geq 1$ and cut the tree at generation $(n-k)$. We use the preliminary bound obtained from \cref{prelimupper} for the displacements until generation $(n-k)$. The upper bound on the displacements occurred in the last $k$ generations is then obtained from Lemma~\ref{upperlemma1}, which gives an upper bound to the maximum position in the last generation of a BRW, started with (possibly random) exponentially growing number of vertices but grown only upto a finite number of generations. Combining them together and taking $k$ tend to infinity will finish our proof.

\newcommand\size{\sigma}
\begin{lemma}
  \label{upperlemma1}
  Fix $k \geq 1$. Consider a probability space $(\Omega, \mathcal{F}, \mathbb{P})$ and a filtration $\left\{\mathcal{H}_k : k \in \mathbb{Z}\right\}$. Start with a sequence of random positive integers
    $\left\{\size_n \right\}_{n \geq 1}$ satisfying  $\liminf_{n \to \infty} n^{-1} \log \log \sigma_n >0$, almost surely.
   Consider the
    following dynamics. Start with $\sigma_n$ particles in $0$-th generation (possibly scattered randomly over the real line) 
    Then grow a BRW, $\mathbb{B}_n$, up to $k$ generations with progeny distribution
    $G$ satisfying \cref{davies_assump} and displacement distribution $F$ satisfying Assumption~\ref{item:F1}. 
    Let $V_{n,q}$ denotes
    the set of vertices in $q$-th generation of $\mathbb{B}_n$ 
    Assume the following conditions.
  \begin{enumerate}
  \item $\mathcal{H}_{n+q}$ contains all the information (cardinality and position of its constituents particles on the real line) of the random set $V_{n,q}$, for all $0 \leq q \leq k$, and $n \geq 1$.
  \item Fix $m \geq 2$. Conditionally on $\mathcal{H}_{m-1}$, $V_{n,m-n}$ is a BRW generation step for progeny $G$ and displacement $F$ from $V_{n,m-n-1}$,  $ \forall \max(m-k,1) \leq n \leq m-1.$ Further, these BRW generation steps are conditionally independent of each other conditional on $\mathcal{H}_{m-1}$. 
  \end{enumerate}
  Let $S_{n,v,q}$ denotes the
  displacement of the particle $v \in V_{n,t}$ from its ancestor in generation $q$, $0 \leq q \leq t \leq k$, for the BRW $\mathbb{B}_n$.
  Then we have the following asymptotic behaviour almost surely.
  \begin{align}
     \label{eq:crucial}
    \limsup_{n \longrightarrow \infty} \frac{\max_{v \in V_{n,k}} S_{n,v,0}}{L(\log \size_n)} \leq \widetilde{\alpha}_k :=
    \begin{cases}  \left(\sum_{i=1}^{k} \alpha^{-i/(r-1)} \right)^{1-1/r},& \mbox{if } r>1, \\
      \alpha^{-k/r}, & \mbox{if } r \in (0,1]. \end{cases}
  \end{align}
\end{lemma}

\begin{remark}
In the statement of \cref{upperlemma1} and its proof, the statement \textit{``Conditional on $\mathcal{G}$, $\mathbb{S}_1$ is a BRW generation step for progeny $G$ and displacement $F$ from $\mathbb{S}$"} means that the $\sigma$-algebra $G$ contains all the information (the size and the position of its constituent particles) for the (random) set $\mathbb{S}$ and $\mathbb{S}_1$ is the first generation of a BRW with progeny $G$ and displacement $F$ having $\mathcal{S}$ as the $0$-th generation, conditional on $\mathcal{G}$. Conditional independence of two or more BRW generation steps will refer to conditional independence of progeny sizes and displacements induced in those steps.
\end{remark}

To explain briefly how the proof of Lemma~\ref{upperlemma1} works, we partition the first generation $V_{n,1}$ according to their displacements from their parents and recursively use similar partitioning technique on the sets of children of each partition block of the first generation. We continue with this recursive partitioning until the last generation. We then provide an upper bound to the total displacement incurred on a path travelling from the initial to last generation by the sum of maximum displacements of each partition block in which its constituent edges lie. Maximizing over such paths followed by optimizing over suitable partitions yield our required upper bound.

\begin{proof}[Proof of Lemma \ref{upperlemma1}]
  We shall prove \cref{eq:crucial} by induction on $k$. As usual, $X_e$ will denote the displacement
  associated with the edge $e$ and $e_v$ denotes the edge connecting the particle $v$ to its parent
  $p(v)$. 
  Note that the sequence $\{\widetilde\alpha_k\}$ satisfies the recursion
  \begin{align}
    \label{eq:recursionalpha}
    \widetilde\alpha_{k+1} 
    & = \begin{cases}
      \alpha^{-1/r} (\widetilde\alpha_k^{r/(r-1)}+1)^{(r-1)/r} , & r > 1 \; , \\
      \alpha^{-1/r}  \widetilde\alpha_k , & 0<r \leq 1\; . \end{cases} 
  \end{align}
  We start by proving the case $k=1$. Since $|V_{n,1}|$ is the sum of $\size_n$ many i.i.d. variables
   having distribution $G$, \cref{strongheavy1} yields
  $\log|V_{n,1}| \stackrel{a.s.}{\sim} (1/\alpha) \log \size_n$, in particular $\liminf_{n \to \infty} n^{-1} \log \log |V_{n,1}| >0,$ almost surely. Here we have applied \cref{strongheavy1} with $\mathcal{F}_n = \mathcal{H}_{n}, \zeta_n = \sigma_n, G_n \equiv G.$ Using independence of progeny and
  displacements, we have by \cref{normordernew} (applied with $\zeta_n=|V_{n,1}|, l_n=\psi_n=1$) that
  $\max_{v \in V_{n,1}} X_{e_v} \stackrel{a.s}{\sim} L( \log |V_{n,1}|)$. Combining these two
  equivalences and using regular variation of $L$, we get
  $$ \max_{v \in V_{n,1}} S_{n,v,0} = \max_{v \in V_{n,1}} X_{e_v} \stackrel{a.s}{\sim} \alpha^{-1/r} L(\log \size_n) = \widetilde\alpha_1 L(\log \sigma_n),$$ proving the
  claim for $k=1$.

  Suppose the claim holds true for $k=1, \ldots, m$, for some $m \geq 1$. We proceed to prove the
  hypothesis for $m+1$. We will classify the particles according to the size of the first
    step. Fix a partition $\Pi = \left\{1=\delta_0 > \delta_1 > \cdots > \delta_l=0 \right\}$ of
  $[0,1]$ and write $\|\Pi\|_r:= \max_{i=1}^l |\delta_i^{1/r}-\delta_{i-1}^{1/r}|$. Consider a
  partition of $V_{n,1}$ as follows. Let $v_{(i)}$ be the particle in $V_{n,1}$ with $i$-th largest
  displacement from its parent among the particles in $V_{n,1}$. The subset $A_{n,i}$ is then defined to
  be consisting of all those particles $v_{(j)}$ with
  $\lfloor |V_{n,1}|^{1-\delta_{i-1}} \rfloor \leq j < \lfloor |V_{n,1}|^{1-\delta_{i}} \rfloor,$
  for all $i=1, \ldots, l$. Keep $v_{(|V_{n,1}|)}$ in the subset $A_{n,l}$. Clearly,
  $\left\{A_{n,i}\right\}_{1 \leq i \leq l}$ forms a partition of $V_{n,1}.$ Notice that
  $|A_{n,i}| \stackrel{a.s.}{\sim} \lfloor |V_{n,1}|^{1-\delta_{i}} \rfloor - \lfloor
  |V_{n,1}|^{1-\delta_{i-1}} \rfloor$ grows double-exponentially almost surely, since
  $\log|V_{n,1}| \stackrel{a.s.}{\sim} (1/\alpha) \log \size_n$. Let $V_{n,m+1,i}$ be the set of
  particles in the $(m+1)$-th generation with first generation ancestor in the subset
  $A_{n,i}$. Fixing our attention only to the subtree generated by elements in $A_{n,i}$, we notice
  this particular subtree to be of height $m$ and satisfies the conditions of \cref{upperlemma1}
  (with $\size_n= |A_{n,i}|$). Therefore, by the induction hypothesis, we have the
  following for all $1 \leq i \leq l$.
  \begin{align*}
    \limsup_{n \to \infty} \frac{\max_{v \in V_{n,m+1,i}} S_{n,v,1}} {L( \log |A_{n,i}|)} \leq \widetilde\alpha_m \; .
  \end{align*}
  Observe that
  $\log |A_{n,i}| \stackrel{a.s.}{\sim} \log \lfloor |V_{n,1}|^{1-\delta_{i}} \rfloor
  \stackrel{a.s.}{\sim} ((1-\delta_i)/\alpha) \log \size_n.$ Combining this with the previous bound,
  we conclude that
  \begin{align*}
    \limsup_{n \to \infty} \frac{\max_{v \in V_{n,m+1,i}} S_{n,v,1}}{L\left( \log \size_n \right)} \leq
    \left(\frac{1-\delta_i}{\alpha}\right)^{1/r} \widetilde\alpha_m \; .
  \end{align*}
  By our construction, $\max_{u \in A_{n,i}} S_{n,u,0} = \max_{u \in A_{n,i}} X_{e_u}$ is the
  $\lfloor |V_{n,1}|^{1-\delta_{i-1}} \rfloor$-th largest among the collection
  $\left\{X_{e_v} : v \in V_{n,1}\right\}.$ Having almost sure  exponential growth of
  $|V_{n,1}|$, we obtain by \cref{normordernew},
  $ \max_{u \in A_{n,i}} S_{n,u,0} \stackrel{a.s.}{\sim} L\left( \delta_{i-1}\log |V_{n,1}| \right)
  \stackrel{a.s.}{\sim} (\delta_{i-1}/\alpha)^{1/r}L\left( \log \size_n \right) $. Combining the
  previous two estimates, observing that
  $$\max_{v \in V_{n,m+1,i}} S_{n,v,0} \leq \max_{v \in V_{n,m+1,i}}S_{n,v,1} + \max_{u \in A_{n,i}}
  S_{n,u,0},$$ 
  and taking minimum over the partitions~$\Pi$, we obtain that almost surely,
  \begin{align*}
    \limsup_{n \to \infty} \frac{\max_{v \in V_{n,m+1}} S_{n,v,0}}{L\left( \log \sigma_n \right) }
    &\leq \alpha^{-1/r} \inf_{\Pi} \left\{ \max_{i=1}^l \left[(1-\delta_i)^{1/r} \widetilde\alpha_m + \delta_{i-1}^{1/r} \right]
      \right\} \\
    &\leq \alpha^{-1/r} \inf_{\Pi} \left\{ \max_{i=1}^l \left[(1-\delta_i)^{1/r} \widetilde\alpha_m + \delta_{i}^{1/r} \right] + \|\Pi\|_r
      \right\} \\
    & \leq 
      \alpha^{-1/r} \left\{ \sup_{s\in [0,1]} \{ (1-s)^{1/r} \widetilde\alpha_m + s^{1/r} \} + \inf_{\Pi}\|\Pi\|_r\right\} \\
    & = \alpha^{-1/r}  \sup_{s\in[0,1]} \{ (1-s)^{1/r} \widetilde\alpha_m + s^{1/r} \} \; .
  \end{align*}
  Solving the optimization and using \cref{eq:recursionalpha}  completes the proof of the lemma.
\end{proof}

\begin{proof}[Proof of \cref{lighttailmain}, upper bound]
  In the proof, all almost sure statement are to be understood as almost surely on $(W>0)$. Fix any $k \geq 1$.  Consider the tree only from generation $(n-k)$. This is a BRW, started with
  $|D_{n-k}|$ vertices and grown up to $k$ generations, with progeny distribution $G$ and
  displacement distribution $F$ satisfying Assumption~\ref{item:F1}. Since $\liminf_{n \to \infty} n^{-1} \log \log |D_{n-k}| = -\log \alpha >0$ almost surely, using \cref{upperlemma1}, we obtain the following upper
  bound.
  \begin{align*}
    \limsup_{n \longrightarrow \infty} \frac{\max_{v \in D_{n}} S_{v,n-k}}{L\left( \log |D_{n-k}| \right)} \leq \alpha_k, \; \text{almost surely},
  \end{align*}
  where $S_{v,n-k}$ denotes the displacement of the vertex $v \in D_n$ from its ancestor in generation $(n-k)$. Here we have applied \cref{upperlemma1} with $\sigma_n = |D_{n-k}|$, $V_{n,q} = D_{n-k+q}$ for $0 \leq q \leq k$ and $\mathcal{H}_n$ being the  $\sigma$-algebra generated by the information of the process upto generation $n-k$.
  Since $\log |D_{n-k}| = \log Z_{n-k} \stackrel{a.s.}{\sim} \alpha^k \log Z_n$, we have
  \begin{align*}
    \limsup_{n \longrightarrow \infty} \frac{\max_{v \in D_{n}} S_{v,n-k}}{L\left( \log Z_n \right)} \leq \alpha^{k/r}\alpha_k
    =  \left(\sum_{t=0}^{k-1} \alpha^{t/(r-1)}\right)^{1-1/r} \vee 1 \; .
  \end{align*}
  On the otherhand, using (\ref{prelimupper}), we get
  \begin{align*}
    \limsup_{n \longrightarrow \infty} \frac{M_{n-k}}{L\left( \log Z_{n-k} \right)} \leq
    \frac{A_2W^{1/r}}{1-\alpha^{1/r}} , \; \text{almost surely}.
  \end{align*}
  Combining the above two estimates with the observation that
  $M_n \leq M_{n-k} + \max_{v \in D_{n}} S_{v,n-k}$, we get, almost surely,
  \begin{align*}
    \limsup_{n \longrightarrow \infty} \frac{M_{n}}{L\left( \log Z_n \right)}
    & \leq \frac{\alpha^{k/r}}{1-\alpha^{1/r}} A_2W^{1/r} + \left[ \left(\sum_{t=0}^{k-1} \alpha^{t/(r-1)}\right)^{1-1/r} \vee 1 \right] \; .
  \end{align*}
  Letting $k$ tend to $ \infty$ concludes the proof.
\end{proof}

\section{BRW with very rapidly varying tails}
\label{section:veryltd}

This section is devoted to the analysis of the situation when the right tail of the displacement distribution function $F$ decays more rapidly than those which satisfy \cref{assumpF}. In particular, it consider the cases where the function $L$, as defined in \cref{defL}, is slowly varying at $\infty$. An example will be $L(x)=\log x$ for all $x >0$, i.e. $F(x)=1-\exp(-e^x)$, for all $x \in \mathbb{R}$. 

We have observed in \cref{th:regvarmain} and \cref{lighttailmain} that, if $F$ satisfies \cref{assumption:displacement_heavy} or \cref{assumpF} with $r \in (0,1]$, only the last generation effectively contributes in determining the right-most position in the $n$-th generation, whereas if $F$ satisfies \ref{assumpF} with $r \in (1, \infty)$, then previous generations also contribute albeit with a geometrically decaying weight. In both cases though, the right-most position in the $n$-th generation is asymptotically of the same order  as the largest displacement incurred in the $n$-th generation. In sharp contrast to these, the case which we shall analyse in this section demonstrates a situation where the right tail of $F$ is so small that each generation contributes equally in determining the right-most position in the last generation. In particular, we expect $M_n$ to be $O \left( \sum_{k=1}^n L(\log Z_k) \right)$, since $L(\log Z_n)$ is the asymptotic order of the largest displacement in the $n$-th generation. If we assume that the right-tail of $F$ decays sufficiently fast, this is indeed the situation as stated in \cref{verylighttail:mainthm}. Also in this situation, $M_n \gg L(\log Z_n)$, a stark difference from those in previous sections.

 We would, though, like to point out that there are some examples of $F$ which lie the twilight zone between those satisfying \cref{assumpF} and \cref{verylighttail:assump}. It may be interesting to try to think of what happens in those cases, but we will not further concern ourselves regarding that matter in this paper.


\begin{assumption}{\label{verylighttail:assump}}
 $F$ is a distribution function on $\mathbb{R}$ with the following properties. Recall that $L$ is the left-continuous inverse of $K=-\log(1-F)$.
  \begin{enumerate}[leftmargin=*,label=(L\arabic*)]
  \item \label{item:L1} $F(x)<1, \;\forall \; x \in \mathbb{R},$ and $L$ is slowly varying at $\infty$ with $K(L(x)) \sim x,$ as $x \to \infty$.
  \item \label{item:L2} There exists $\eta_0 \in (0, 1-\alpha)$ such that for any $\eta \in (0, \eta_0)$, there exists a sequence of non-negative real numbers $\left\{\kappa_n(\eta)\right\}_{n \ge 1}$ satisfying the following conditions.
  $$\lim_{\eta \downarrow 0} \liminf_{n \to \infty} \dfrac{L\left(\alpha^{-n}e^{-\kappa_n(\eta)}\right)}{L\left(\alpha^{-n} \right)} \geq 1, \; \; \text{ and } \sum_{n \geq 1} e^{-\kappa_n(\eta)} < \infty, \; \forall \; \eta \in (0, \eta_0).$$
  
  
  \end{enumerate}
\end{assumption}

\begin{theorem}{\label{verylighttail:mainthm}}
Let \cref{davies_assump} and \cref{verylighttail:assump} hold. Then almost surely conditional on survival of the tree, we have 
$$ \lim_{n \to \infty} \dfrac{M_n}{\sum_{k=1}^n L(\alpha^{-k})} = 1.$$
\end{theorem}

\Cref{verylighttail:mainthm} confirms the intuition described above. Heuristically, this means that the maximum is achieved by a path such that at each generation, the displacement is essentially the maximum of the displacements of that generation.

Before diving into the proof of \cref{verylighttail:mainthm}, let us first make some sense of the complicated and apparently artificial conditions presented in \cref{verylighttail:assump}. It is evident from \cref{item:L1} that $L$ is non-decreasing, slowly varying at $\infty$ with $L(\infty)=\infty$. We shall now explore some simpler assumptions on $L$ which guarantees \cref{item:L2}. In the process, we shall explore some examples of displacement distribution $F$ which satisfies these conditions. 

\begin{corollary}{\label{case1L}}
Assume that $L$ satisfies \cref{item:L1} and further suppose that there exists $a>1$ such that $ x \mapsto L(a^x)$ is regularly varying at $\infty$ with index $ \gamma \geq 0$. Then $L$ satisfies \cref{verylighttail:assump} and we have
$$ \lim_{n \to \infty} \dfrac{M_n}{nL(\alpha^{-n})} = \dfrac{1}{\gamma+1}, \; \text{ almost surely on survival}.$$  
\end{corollary}

\begin{proof}
We shall start by establishing the fact that $x \mapsto L(b^x)$ is regularly varying at $\infty$ with index $\gamma$, for any $b >1$. Fix $t >0$. Then 
$$ \lim_{x \to \infty} \dfrac{L(b^{tx})}{L(b^x)} = \lim_{x \to \infty} \dfrac{L\left(a^{tx \log b /\log a}\right)}{L\left(a^{x\log b /\log a}\right)} = t^{\gamma},$$
establishing our claim. To establish \cref{item:L2}, we make the choice $\kappa_n(\eta)=-n \log (1-\eta)$ for all $n \geq 1$ and $0<\eta < \eta_0 = 1-\alpha.$ Summability of the sequence $\left\{e^{-\kappa_n(\eta)}\right\}_{n \geq 1}$ is obvious.
Using regularly varying property of $x \mapsto L(\alpha^{-x})$, we obtain
$$ \lim_{n \to \infty} \dfrac{L\left(\alpha^{-n}e^{-\kappa_n(\eta)}\right)}{L\left(\alpha^{-n} \right)} = \lim_{n \to \infty} \dfrac{L\left(e^{n (\log (1-\eta) - \log \alpha) }\right)}{L\left(e^{-n \log \alpha} \right)} = \left(\dfrac {\log \alpha - \log (1-\eta)}{\log \alpha} \right)^{\gamma},$$
and hence 
$$ \lim_{\eta \downarrow 0} \lim_{n \to \infty} \dfrac{L\left(\alpha^{-n}e^{-\kappa_n(\eta)}\right)}{L\left(\alpha^{-n} \right)} =1 .$$
These prove that $L$ satisfies \cref{verylighttail:assump} and hence we can apply \cref{verylighttail:mainthm}. All we have to show now is the fact that 
\begin{equation}{\label{toshow1}}
 \sum_{k=1}^n L(\alpha^{-k}) \sim (\gamma+1)^{-1}nL(\alpha^{-n}), \; \text{ as } n \to \infty.
\end{equation}
In order to do so, we first apply monotonicity of $L$ to conclude that 
\begin{equation}{\label{karam}}
(\gamma+1)^{-1}nL(\alpha^{-n}) \sim \int_{0}^n L(\alpha^{-x}) \, dx \leq \sum_{k=1}^n L(\alpha^{-k}) \leq \int_{0}^{n+1} L(\alpha^{-x}) \, dx \sim  (\gamma+1)^{-1}(n+1)L(\alpha^{-n-1}),
\end{equation}
where the first and the last asymptotics in \cref{karam} follows from \textit{Karamata's Theorem} (see~\cite[Theorem 0.6]{Resnick}) and the fact that $x \mapsto L(\alpha^{-x})$ is regularly varying with index $\gamma \geq 0$. This regularly varying property, along with monotonicity, also implies that $nL(\alpha^{-n}) \sim (n+1)L(\alpha^{-n-1})$ and thus \cref{toshow1} is established.  
\end{proof}

\begin{example}
The condition on $L$ assumed in \cref{case1L} is satisfied for all the distribution functions $F$ for which $K=-\log(1-F)$, for large enough $x$, has the form $K(x) = \exp(cx^{\beta})$, for some $c, \beta >0$ or $K(x) = \exp(c_1e^{c_2x})$, for some $c_1,c_2>0$. In the first case, $L(x) = \left(\log (x)/c \right)^{1/\beta}$ and hence satisfies \cref{case1L} with $\gamma = 1/\beta$. In the second set of examples, clearly $\gamma=0$.
\end{example}

The next set of examples of $L$, that we shall consider, will allow for $L$ to be increasing in a much slower rate than required by \cref{case1L}.

\begin{corollary}{\label{caseL2}}
Suppose that $L$ satisfies \cref{item:L1} and there exists $\left\{\kappa_n : n \geq 1\right\}$, positive real numbers, such that
\begin{equation}{\label{cond}}
 L(\alpha^{-n}e^{-\kappa_n}) \sim L(\alpha^{-n}), \; \text{ as } n \to \infty, \; \text{ and } \sum_{n \geq 1} \exp(-\kappa_n) < \infty.
\end{equation}
 Then $L$ satisfies \cref{verylighttail:assump} and therefore the conclusion of \cref{verylighttail:mainthm} holds true.
\end{corollary} 

\begin{proof}
The statement is immediate from \cref{verylighttail:mainthm} if we take $\kappa_n(\eta)=\kappa_n$, for all $n \geq 1$ and $\eta \in (0,1-\alpha)$.
\end{proof}

The $L$ in Assumption \ref{item:L1} is a slowly varying function and for such a function we have a Karamata's representation theorem (see \cite[Section 1.3]{bingham} for more details on Karamata Representation). We now impose some conditions on this representations that allows us to verify \cref{verylighttail:assump}.
\begin{corollary}{\label{case4L}}
Suppose that $L$ satisfies \cref{item:L1} and the following is a Karamata Representation of $L$ for some $a>0$.
$$ L(x) = c(x) \exp \left( \int_{a}^x \dfrac{\varepsilon(t)}{t} \, dt\right), \; \forall \; x >a, $$
where $c,\varepsilon : \mathbb{R}^+ \mapsto \mathbb{R}^{+}$ and $\lim_{x \to \infty} c(x)=c \in (0, \infty), \lim_{t \to \infty} \varepsilon(t)=0.$  Further assume that $\varepsilon(t)\log \log t =o(1)$, as $t \to \infty$. Then $L$ satisfies \cref{verylighttail:assump} and therefore the conclusion of \cref{verylighttail:mainthm} holds true.
\end{corollary}

\begin{proof}
We shall show that $L$ satisfies the conditions required to apply \cref{caseL2} for the choice $\kappa_n=2\log n$. Clearly, $\sum_{n \geq 1} \exp(-\kappa_n) < \infty.$ Setting 
$$ \upsilon_n := \left(\log \log \alpha^{-n} \right) \sup_{t \in \left[L(\alpha^{-n}e^{-\kappa_n}),L(\alpha^{-n}) \right]} \varepsilon(t) = o(1), $$
we note that for large enough $n$,
\begin{align*}
\dfrac{L(\alpha^{-n}e^{-\kappa_n})}{L(\alpha^{-n})} = \dfrac{c(\alpha^{-n}e^{-\kappa_n})}{c(\alpha^{-n})}\exp \left( - \int_{L(\alpha^{-n}e^{-\kappa_n})}^{L(\alpha^{-n})} \dfrac{\varepsilon(t)}{t} \, dt \right) & \geq \dfrac{c(\alpha^{-n}e^{-\kappa_n})}{c(\alpha^{-n})}\exp \left( - \dfrac{\upsilon_n}{\log \log \alpha^{-n}} \log \dfrac{\alpha^{-n}}{\alpha^{-n}e^{-\kappa_n}}  \right) \\
& =  \dfrac{c(\alpha^{-n}e^{-\kappa_n})}{c(\alpha^{-n})}\exp \left( - \dfrac{2\upsilon_n \log n }{\log n + \log \log \alpha^{-1}}   \right),
\end{align*}
and hence $L(\alpha^{-n}e^{-\kappa_n}) \sim L(\alpha^{-n}).$ This concludes the proof.
\end{proof}

\begin{corollary}{\label{case5L}}
Suppose that $L$ satisfies \cref{item:L1} and of the form $L(x)=\exp(h(\log x)+o(1))$ as $x \to \infty$ where $h$ is   differentiable for all $x >x_0$ and satisfies $\lim_{x \to \infty} h^{\prime}(x)\log x =0$. 
Then $L$ satisfies \cref{verylighttail:assump} and therefore the conclusion of \cref{verylighttail:mainthm} holds true. A sufficient condition implying $\lim_{x \to \infty} h^{\prime}(x)\log x =0$ is that $h$ is regularly varying with index $\rho \in [0,1)$ with $h^{\prime}$ being ultimately monotone.
\end{corollary}

\begin{proof}
Assuming without loss of generality that $L(x_0) >0$,  we have  a \textit{Karamata Representation} of $L$ as follows.
\begin{align*}
L(x) = c(x) \exp \left( h(\log x) - h(x_0)\right) = c(x) \exp \left( \int_{x_0}^{\log x} h^{\prime}(t)\, dt \right) = c(x) \exp \left( \int_{\log x_0}^{x} \dfrac{h^{\prime}(\log u )}{u}\, du \right)
\end{align*}
where $c(x) \to e^{h(x_0)}$ as $x \to \infty$. Application of \cref{case4L}, along with the condition that $h^{\prime}(x)\log x=o(1)$, implies the first assertion.

If $h$ is regularly varying with index $\rho \in [0,1)$ and $h^{\prime}$ is ultimately monotone, we can apply \textit{Monotone Density Theorem}~\cite[Theorem 1.7.2]{bingham} to guarantee that $h^{\prime}(x) \sim \rho x^{\rho-1} l(x)$ as $x \to \infty$, for some slowly varying function $l$. Clearly,
$$ h^{\prime}(x) \log x = \dfrac{x^{1-\rho}h^{\prime}(x)}{l(x)} \dfrac{l(x)\log x}{x^{1-\rho}} \to 0, \; \text{ as } x \to \infty,$$
since $x \mapsto l(x)\log x$ is also slowly varying and $\rho <1$. This proves the second assertion.
\end{proof}

\begin{example}
A class of distribution functions $F$ for which the criteria stated in \cref{case5L} holds is as follows. We take $K(x) = e^{g(\log x)}$, where $g$ is regularly varying with index $\beta >1$, with monotone derivative $g^{\prime}$ which is then regularly varying with index $\beta-1>0$. Then $L(x) = e^{h(\log x)}$ with $h=g^{\leftarrow}$, the inverse of $g$, regularly varying with index $1/\beta$. It is easy to see that members of this class do not satisfy the conditions stated in \cref{case1L}, but satisfy the conditions stated in \cref{case5L}. Moreover, following similar arguments as presented while establishing \cref{karam}, we can show that in this case,
\begin{align*}
\sum_{k=1}^n L(\alpha^{-k}) \sim \int_{0}^n L(\alpha^{-t}) \, dt = \int_{0}^n  e^{h(t \log (1/\alpha))} \, dt & = \dfrac{1}{\log (1/\alpha)} \int_{h(0)}^{h(n \log (1/\alpha))}  g^{\prime}(u) e^u \, du \\
& \sim \dfrac{1}{\log (1/\alpha)} g^{\prime} \left(h(n \log (1/\alpha)) \right) e^{h(n \log (1/\alpha))} \\
& \sim \log ^{-1/\beta}(1/\alpha) g^{\prime}(h(n)) L(\alpha^{-n}).
\end{align*}
The last asymptotic equivalence is due to the fact that $g^{\prime}\circ h$ is regularly varying with index $1-1/\beta$. Therefore, almost surely on $(W>0)$,
\begin{align*}
\lim_{n \to \infty} \dfrac{M_n}{g^{\prime}(g^{\leftarrow}(n))L(\alpha^{-n})} = \dfrac{1}{\log ^{1/\beta}(1/\alpha)}.
\end{align*}
For instance, taking $g(x) = cx^{\beta}$ with $\beta >1$ and setting $\tau=(\log(1/\alpha)/c)^{1/\beta}$ yields
\begin{align*}
\lim_{n \to \infty} \dfrac{M_n}{n^{1-1/\beta}\exp(\tau n^{1/\beta})} = \dfrac{\beta}{\tau}, \; \text {almost surely on } (W>0).
\end{align*}
\end{example}

Before delving into the details, we briefly describe the strategy we are going to adopt to prove \cref{verylighttail:mainthm}.
 The upper bound to the maximum is quite straight-forward, as the right-most position is always bounded above by the sum of the maximum displacements over generations. In order to establish the lower bound, we extract a subtree with $n$-th generation be denoted by $A_n$ which consists of those particles whose displacements from their parents are approximately among the largest $O\left(|B_{n}|^{1-\delta_n}\right)$ many, where $B_n$ is the set of children of $A_{n-1}$ in the original tree.  Choosing the sequence $\left\{\delta_n \right\}$ to be converging to $0$ at an appropriate speed as dictated by \cref{item:L2}, we can  in effect guarantee  the existence of an infinite ray, starting from the root, along which displacements from the parents are among the largest in the corresponding generations.  This obviously gives a matching lower bound.

\begin{proof}[Proof of upper bound in \cref{verylighttail:mainthm}]
Fix $\varepsilon >0$.   Note that $M_n\le \sum_{i=1}^n \widetilde M_{Z_i}$. We get the following for all large enough $n$ almost surely.
   \begin{align*}
      \mathbb{P} \left[ \widetilde{M}_{Z_i} > L((1+\varepsilon)\log Z_i)\Big \vert \mathbb{T}\right]
      & = 1- \left(F \circ L((1+\varepsilon)\log Z_i)\right)^{Z_i} \\
      &= 1- \left(1- \exp(-K\circ L((1+\varepsilon)\log Z_i))\right)^{Z_i} \\
      & \leq 1- \left(1- \exp(-(1+\varepsilon)\log Z_i)\right)^{Z_i} \\
      & = 1- \left(1-  Z_i^{-1-\varepsilon}\right)^{Z_i} \leq Z_i^{-\varepsilon} \; .
    \end{align*}
    Since the last expression is summable, we have by Borel-Cantelli Lemma, 
    almost surely conditioned on survival,
    \begin{equation}\label{eq:upperbdM-tildeM2}
       M_n\le \sum_{i=1}^n \widetilde M_{Z_i} \le \sum_{i=1}^n L((1+\varepsilon)\log Z_i)+ \sum_{i=1}^{\infty} \left( \widetilde{M}_{Z_i} \vee 0 \right) \mathbbm{1} \left(  \widetilde{M}_{Z_i} > L((1+\varepsilon)\log Z_i)\right), 
    \end{equation}
    where the second expression in the right hand side of (\ref{eq:upperbdM-tildeM2}) is finite almost surely.
  Since the last expression is summable, we have by Borel-Cantelli Lemma, 
  \begin{align*}
  \limsup_{n \to \infty} \dfrac{M_n}{\sum_{k=1}^n L(\alpha^{-k})} \leq  \limsup_{n \to \infty}  \dfrac{\sum_{k=1}^n L((1+\varepsilon)\log Z_k)}{\sum_{k=1}^n L(\alpha^{-k})} &\leq \limsup_{n \to \infty} \dfrac{ L((1+\varepsilon)\log Z_n)}{ L(\alpha^{-n})} =1,
  \end{align*}
  almost surely on $(W>0)$. The last limit follows from slowly varying property of $L$ and the fact that $(1+\varepsilon)\alpha^n \log Z_n \to (1+\varepsilon)W$, almost surely.
  \end{proof}
 
  

\begin{proof}[Proof of lower bound in \cref{verylighttail:mainthm}]
Fix $\eta \in(0, \eta_0)$, such that $L(\alpha^{-j}\exp(-\kappa_j(\eta))) \to \infty$, as $j \to \infty$. Indeed, the first condition in \cref{item:L2} and the fact that $L(\infty)=\infty$ guarantee that the previously mentioned property holds true for all small enough $\eta$. 
Set $c(\eta)=\sum_{n \geq 1} \exp(-\kappa_n(\eta)) < \infty$. Fix $\upsilon \in (0, 1-\alpha)$.  Define the following sequences.
$$ \gamma_n(\eta) := 1- \dfrac{1}{2c(\eta)} \sum_{k=1}^n \exp \left( -\kappa_k(\eta)\right) , \; \delta_n(\eta):= \dfrac{\upsilon(\gamma_{n-1}(\eta)-\gamma_n(\eta))}{\gamma_{n-1}(\eta)}, \; \forall \; n \geq 1,$$
with $\gamma_0(\eta):=1$. By construction, $\delta_n(\eta) \in (0,\upsilon)$, $\gamma_n(\eta) =\prod_{k=1}^n \left(1-\upsilon^{-1}\delta_k(\eta)\right) \in [1/2,1]$ and hence
\begin{align}
e^{\kappa_n(\eta)}\delta_n(\eta) \prod_{k=1}^{n-1}\left(1-\delta_k(\eta)\right) &> e^{\kappa_n(\eta)}\delta_n(\eta) \prod_{k=1}^{n-1}\left(1-\upsilon^{-1}\delta_k(\eta)\right)  \nonumber\\
&=e^{\kappa_n(\eta)}\upsilon \left( \gamma_{n-1}(\eta)-\gamma_n(\eta)\right) = \dfrac{\upsilon}{2c(\eta)} >0, \; \forall \; n \geq 1. \label{bound}
\end{align} 
Recursively define the following sets of vertices. Set $B_{0, \eta}=A_{0, \eta} :=D_0$, which contains only the root. For any $j \geq 1$, set $B_{j,\eta}$ to be the set children of the vertices in $A_{j-1, \eta}$. Recall the notation $C(v)$ which denotes the set of children for any particle $v$. Define $A_{j,\eta}$ to be the those particles in $B_{j,\eta}$ whose displacement from their parents are among the maximum $|A_{j,\eta}|$ many of those in $B_{j,\eta}$; break the ties uniformly at random if needed. Here
$$ |A_{j,\eta}| := \sum_{v \in A_{j-1,\eta}} \lfloor |C(v)|^{1-\delta_j(\eta)} \rfloor.$$
It is immediately observed that the sequence $\left\{|A_{j,\eta}| : j \geq 0\right\}$ has the same law as $\left\{Z_{j,\eta}^* : j \geq 0\right\}$, where $Z_{j,\eta}^*$ is the size of the $j$-th generation of the inhomogeneous branching process, started from one particle, and where the particles in $n$-th generation produces i.i.d. many off-springs (independent of the structure of the tree until that generation) having distribution $G_{(1-\delta_{n+1}(\eta))}$, as defined before \cref{thin1}. Since $\delta_j(\eta) \in (0, \upsilon)$ for $\upsilon < 1-\alpha$, we apply \cref{inhom} and \cref{thin1} to obtain the following.
$$ \mathbb{P} \left[\lim_{n \to \infty}  \left(\prod_{j=1}^n  \dfrac{\alpha}{1-\delta_j(\eta)} \right) \log |A_{n,\eta}| = W^{*} > 0 \right] =: q(\eta) >0.$$
Let $E_{\eta}$ denote the event in the above equation. In particular, almost surely on $E_{\eta}$,
$$ \liminf_{n \to \infty} n^{-1}\log \log |A_{n,\eta}| \geq \liminf_{n \to \infty} \dfrac{1}{n}\sum_{k=1}^n \log \left( \dfrac{1-\delta_j(\eta)}{\alpha} \right) \geq \log \left( \dfrac{1-\upsilon}{\alpha} \right) >0. $$
A standard application of \cref{strongheavy1} then implies that $\log |B_{j,\eta}| \stackrel{a.s.}{\sim} \alpha^{-1}\log |A_{j-1,\eta}|$ as $j \to \infty$, on $E_{\eta}$. Here we have applied \cref{strongheavy1} with $\mathcal{F}_j$ being the $\sigma$-algebra containing all the information (tree structure and displacement of the particles) upto generation $j$, $\zeta_j = |A_{j,\eta}|$ and $G_j = G$ for all $j \geq 0$. In particular, $|B_{j,\eta}|$ also grows double-exponentially almost surely on $E_{\eta}$. Another similar application of \cref{strongheavy1} and \cref{thin1} guarantees that  $\log |A_{j,\eta}| \stackrel{a.s.}{\sim} (1-\delta_j(\eta))\alpha^{-1}\log |A_{j-1,\eta}|$ as $j \to \infty$,  on $E_{\eta}$. Here we have applied \cref{strongheavy1} with the same choice of $\mathcal{F}_j$ and $\zeta_j$ as earlier but with $G_j=G_{(1-\delta_{j+1}(\eta))}$, for all $j \geq 0$. 
Combining these two, we have  $\log |A_{j,\eta}| \stackrel{a.s.}{\sim} (1-\delta_j(\eta))\log |B_{j,\eta}|$ as $j \to \infty$, on $E_{\eta}$.

\cref{item:L1} guarantees that $L$ satisfies the assumptions required to apply \cref{normordernew}. Since, on the event $E_{\eta}$, $|B_{j,\eta}|$ grows double-exponentially almost surely, we can employ \cref{normordernew} to yield the following.
\begin{equation}{\label{minconv}}
\dfrac{\min_{v \in A_{j, \eta}} X_{e_v}}{L(\delta_j(\eta) \log |B_{j,\eta}|)} \stackrel{a.s.}{\longrightarrow} 1, \; \text{on } E_{\eta}.
\end{equation}
The application of \Cref{normordernew} is valid indeed, since on the event $E_{\eta}$, 
$$ \delta_j(\eta) \log |B_{j, \eta}| \stackrel{a.s.}{\sim} \dfrac{\delta_j(\eta)}{1-\delta_j(\eta)} \log |A_{j, \eta}| \stackrel{a.s.}{\sim} \dfrac{\delta_j(\eta)}{1-\delta_j(\eta)} W^{*} \alpha^{-j}\prod_{i=1}^j (1-\delta_i(\eta)) \geq W^*\alpha^{-j}e^{-\kappa_j(\eta)} \dfrac{\upsilon}{2c(\eta)} \to \infty,$$
where the last line follows from the fact that $L(\alpha^{-j}e^{-\kappa_j(\eta)}) \to \infty$ and $K(x) < \infty$ for all $x$ with $K(\infty)=\infty$.

By the slowly varying property of $L$, we also have for large enough $j$,
\begin{align}
L(\delta_j(\eta) \log |B_{j,\eta}|) \stackrel{a.s.}{\sim} L\left(\dfrac{\delta_j(\eta)}{1-\delta_j(\eta)} \log |A_{j,\eta}| \right) &  \stackrel{a.s.}{\sim}  L\left(\dfrac{W^*\delta_j(\eta)}{1-\delta_j(\eta)} \prod_{k=1}^j \dfrac{1-\delta_k(\eta)}{\alpha}\right) \nonumber \\
& \geq L \left(\alpha^{-j}W^{*} \dfrac{\upsilon e^{-\kappa_j(\eta)}}{2c(\eta)} \right) \stackrel{a.s.}{\sim} L \left(\alpha^{-j} e^{-\kappa_j(\eta)} \right) \label{minconv2},
\end{align} 
on $E_{\eta}$. Here we have used the definition of $E_{\eta}$, monotonicity of $L$ and \cref{bound} to obtain the  inequality above. Therefore,
\begin{align}{\label{finalconv}}
\lim_{j \to \infty} \dfrac{\min_{v \in A_{j, \eta}} X_{e_v}}{L(\alpha^{-j}e^{-\kappa_j(\eta)})} = 1, \; \text{almost surely on } E_{\eta}.
\end{align}
By construction,
$$ M_n \mathbbm{1}(|A_{n, \eta}| >0) \geq \mathbbm{1}(|A_{n, \eta}|>0) \max_{v \in A_{n, \eta}} S_v \geq  \mathbbm{1}(|A_{n, \eta}|>0) \sum_{j=1}^n  \min_{v \in A_{j, \eta}} X_{e_v}.$$ The logic behind the above inequality is same as the one used in the proof of lower bound in \cref{lighttailmain}. Since, $|A_{n,\eta}|$ converges to infinity almost surely on $E_{\eta}$, we can apply \cref{finalconv} and \cref{cezaro} to conclude that 
$$ \liminf_{n \to \infty} \dfrac{M_n}{\sum_{j=1}^n L(\alpha^{-j})}\geq  \liminf_{n \to \infty} \dfrac{L\left(\alpha^{-n}e^{-\kappa_n(\eta)}\right)}{L\left(\alpha^{-n} \right)} =:\varphi(\eta), \; \text{almost surely on } E_{\eta}.$$


Set $c_n := \sum_{j=1}^{n} L(\alpha^{-j})$, for all $n \geq 1$. Monotonicity and slowly varying property of $L$ and \cref{cezaro} implies that $c_n \sim c_{n-k}$, as $k \to \infty$, for any $k \in \mathbb{N}$. Fix any $k \geq 1$ and let $S_{v,n,n-k}$ be the displacement of particle $v \in D_n$ from its ancestor in generation $k$. For any $u \in D_k$, let $D_{u,n,n-k}$ be the particles in $D_n$ with $u$ as an ancestor.  Our previous analysis shows that, conditioned on $D_k$, there exists independent events  $\left\{E_{\eta,u} : u \in D_k\right\}$, having same probability as $E_{\eta}$, such that 
 $$ \liminf_{n \to \infty} c_{n}^{-1} M_{u,n,n-k}=\liminf_{n \to \infty} c_{n-k}^{-1} M_{u,n,n-k} :=  \liminf_{n \to \infty} c_{n-k}^{-1} \max_{v \in D_{u,n,n-k}} S_{v,n,n-k} \geq \varphi(\eta), \; \text{ a.s. on } E_{\eta,u}.$$
 Since, $M_n = \max_{u \in D_k} \left(S_u + M_{u,n,n-k} \right)$, we can argue that
 \begin{align*}
 \mathbb{P} \left( \liminf_{n \to \infty} c_n^{-1}M_n < \varphi(\eta) \big \rvert \mathbb{T}\right) = \lim_{k \to \infty} \mathbb{P} \left( \liminf_{n \to \infty} c_n^{-1}M_n < \varphi(\eta) \big \rvert \left\{D_j\right\}_{1 \leq j \leq k}\right) & = \lim_{k \to \infty} \mathbb{P} \left( \bigcap_{u \in D_k} E_{\eta,u}^c  \big \rvert \left\{D_j\right\}_{1 \leq j \leq k} \right) \\
 & = \lim_{k \to \infty}(1-q({\eta}))^{Z_k} \\
 &= \mathbbm{1}(\mathbb{T} \text{ extincts }) = \mathbbm{1}(W=0).
 \end{align*}
Therefore, almost surely on the event $(W>0)$, we have $\liminf_{n \to \infty} c_n^{-1}M_n \geq \varphi(\eta)$.Taking $\eta \downarrow 0$ and using the first condition in \cref{item:L2}, yields $\liminf_{n\to\infty} c_n^{-1}M_n\geq1$. Since we have already proved that the corresponding limsup is less than or equal to 1, this concludes the proof.
\end{proof}

\section{Speed of BRW with infinite progeny mean}{\label{sec:speed}}

There exists different ways to define a speed for the BRW. In \cite{benjperes,lyonspem,peres}, three notions of speed were introduced; namely \textit{Cloud speed, Burst speed} and \textit{Sustainable speed}. In order to define them, we introduce the following notations. Rays of the tree $\mathbb{T}$ are formally defined as infinite paths starting from the root which do not backstep. We denote the rays of the tree $\mathbb{T}$ by $\mathcal{R}$ and the set of all rays by $\partial \mathbb{T}$. The aforementioned three notions are defined as follows.
\begin{equation}{\label{cloud}}
\text{Cloud speed} : \; s_{\text{cloud}} := \limsup_{n \to \infty}  \max_{v : |v| =n } \dfrac{S_v}{|v|} = \limsup_{n \to \infty} \dfrac{M_n}{n};
\end{equation}
\begin{equation}{\label{bursssust}}
\text{Burst speed} : \; s_{\text{burst}} := \sup_{\mathcal{R} \in \partial \mathbb{T}} \limsup_{v \in \mathcal{R}} \dfrac{S_v}{|v|}, \; \;\text{Sustainable speed} : \; s_{\text{sust}} := \sup_{\mathcal{R} \in \partial \mathbb{T}} \liminf_{v \in \mathcal{R}} \dfrac{S_v}{|v|}.
\end{equation}
Here $|v|$ denotes the distance of the vertex $v$ from the root, i.e., the generation to which $v$ belongs. We refer to \cite{peres} for a detailed exposition on these concepts. By Kolmogorov $0-1$ law the speeds are almost surely constant when the displacements are i.i.d, though they might be different from each other. The following relation holds trivially.
\begin{equation}{\label{order}}
s_{\text{cloud}} \geq s_{\text{burst}} \geq s_{\text{sust}}.
\end{equation}
 It is a well-known fact that for BRW the notions of cloud speed, burst speed and sustainable speed coincide, established by \cite{Hammersley, Kingman, Biggins}. This statement was proved under the assumptions that the progeny variables have finite mean and the displacement variables  have finite moment generating function. The later condition was removed by \cite{gantert} for displacement variables with semi-exponential tails, accommodating the changing rate of growth in the definition of the speeds. Since the rate of growth for the maximum displacement are drastically different from one another for the cases we considered in \Cref{heavytail}, \Cref{section:ltd} and \Cref{section:veryltd}; we need to properly modify our definition of speeds in those cases. In all  these cases our target is to choose a correct rate of growth to get  almost surely constant finite and non-zero speed, which is same for all three notions. The following three results serve this purpose. In \Cref{speedh,speedl,speedvl} below, the three terms correspond to cloud speed, burst speed and sustainable speed, respectively, for the corresponding cases. They are also trivially in non-increasing order similar to \cref{order}.
 \begin{proposition}{\label{speedheavy}}
 Under \Cref{davies_assump} and \Cref{assumption:displacement_heavy}, the following holds almost surely conditional on the survival of the tree.
 \begin{equation}{\label{speedh}}
 \limsup_{n \to \infty}  \max_{v : |v| =n } \dfrac{\log \log S_v}{|v|} = \sup_{\mathcal{R} \in \partial \mathbb{T}} \limsup_{v \in \mathcal{R}} \dfrac{\log \log S_v}{|v|}= \sup_{\mathcal{R} \in \partial \mathbb{T}} \liminf_{v \in \mathcal{R}} \dfrac{\log \log S_v}{|v|}= - \log \alpha.
 \end{equation}  
 \end{proposition}
 \begin{proposition}{\label{speedlighttail}}
 Under \Cref{davies_assump} and \Cref{assumpF}, the following holds almost surely conditional on the survival of the tree.
  \begin{equation}{\label{speedl}}
  \limsup_{n \to \infty}  \max_{v : |v| =n } \dfrac{ \log S_v}{|v|} = \sup_{\mathcal{R} \in \partial \mathbb{T}} \limsup_{v \in \mathcal{R}} \dfrac{ \log S_v}{|v|}= \sup_{\mathcal{R} \in \partial \mathbb{T}} \liminf_{v \in \mathcal{R}} \dfrac{ \log S_v}{|v|}= - \dfrac{1}{r}\log \alpha. 
  \end{equation}
 \end{proposition}
 \begin{proposition}{\label{speedverylight}}
  Under \Cref{davies_assump} and \Cref{verylighttail:assump}, the following holds almost surely conditional on the survival of the tree.
  \begin{equation}{\label{speedvl}}
  \limsup_{n \to \infty}  \max_{v : |v| =n } \dfrac{ S_v}{\sum_{k=1}^{|v|} L(\alpha^{-k})} = \sup_{\mathcal{R} \in \partial \mathbb{T}} \limsup_{v \in \mathcal{R}} \dfrac{S_v}{\sum_{k=1}^{|v|} L(\alpha^{-k})}= \sup_{\mathcal{R} \in \partial \mathbb{T}} \liminf_{v \in \mathcal{R}} \dfrac{S_v}{\sum_{k=1}^{|v|} L(\alpha^{-k})}=1.
  \end{equation}
  \end{proposition}
 
Note that for heavy tailed displacements, the speed is of linear order in log-log scale; whereas for rapidly decaying displacements satisfying \cref{assumpF}, the speed is of linear order in log-scale. For very rapidly decaying displacements, the speed depends on the nature of the function $L$.  As for example, under \cref{case1L} the speed is of some polynomial order.

The proofs of these three results are very similar in flavour. The upper bound on the corresponding notion of the cloud speed follows from the asymptotics of the maximum displacement, investigated in the previous sections; whereas lower bound on the sustainable speed follows from constructing a ray along which most of the displacements are very large. This construction is very similar to what we did in the proof of lower bound in \Cref{verylighttail:mainthm}.
 
\begin{proof}[Proof of \Cref{speedheavy}]
Taking logarithm on both sides of the main assertion of \cref{aslim:regtail}, we obtain $ n^{-1} \log \log M_n \stackrel{a.s.}{\longrightarrow} -\log \alpha$. It is therefore enough to show that 
$$ \sup_{\mathcal{R} \in \partial \mathbb{T}} \liminf_{v \in \mathcal{R}} \dfrac{\log \log S_v}{|v|} \geq  - \log \alpha. $$
Fix $\delta \in (0, 1-\alpha)$ and recursively define the following sets of vertices. Set $B_{0}=A_{0} :=D_0$, which contains only the root. For any $j \geq 1$, set $B_{j}$ to be the set children of the vertices in $A_{j-1}$. Define $A_{j}$ to be the those particles in $B_{j}$ whose displacement from their parents are among the maximum $|A_{j}|$ many of those in $B_{j}$; break the ties uniformly at random if needed. Here
$$ |A_{j}|:= \sum_{v \in A_{j-1}} \lfloor |C(v)|^{1-\delta} \rfloor.$$
It is immediately observed that the sequence $\left\{|A_{j}| : j \geq 0\right\}$ has the same law as $\left\{\widetilde{Z}_{j} : j \geq 0\right\}$, where $\widetilde{Z}_{j}$ is the size of the $j$-th generation of the homogeneous branching process, started from one particle, and where the particles in $n$-th generation produces i.i.d. many off-springs (independent of the structure of the tree until that generation) having distribution $G_{(1-\delta)}$, as defined before \cref{thin1}. We apply \cref{davies} to obtain the following.
$$ \mathbb{P} \left[\lim_{n \to \infty}  \left( \dfrac{\alpha}{1-\delta} \right)^n \log |A_{n}| = \widetilde{W} > 0 \right] =: \tilde{q} >0.$$
Let $\tilde{E}$ denote the event in the above equation. 
A standard application of \cref{strongheavy1} implies that $\log |B_{j}| \stackrel{a.s.}{\sim} \alpha^{-1}\log |A_{j-1}| \stackrel{a.s.}{\sim} (1-\delta)^{-1}\log |A_{j}|$ as $j \to \infty$, on $\tilde{E}$. We can then employ \cref{normordernewreg} to yield the following.
\begin{equation}{\label{minconv1}}
\min_{v \in A_{j}} \log X_{e_v} \stackrel{a.s.}{\sim}\delta \beta^{-1} \log |B_{j}| \stackrel{a.s.}{\sim} \delta(1-\delta)^{-1}\beta^{-1} \log |A_j| \stackrel{a.s.}{\sim} \delta\beta^{-1}(1-\delta)^{j-1} \alpha^{-j}\widetilde{W}, \; \text{on } \tilde{E},
\end{equation}
and since $1-\delta > \alpha$, we have
$$ \max_{j=1}^k \min_{v \in A_{j}} \log X_{e_v} \stackrel{a.s.}{\sim} \delta \beta^{-1}(1-\delta)^{k-1} \alpha^{-k}\widetilde{W}, \; \text{on } \tilde{E}. $$
Now take any infinite ray $\mathcal{R}^*$ whose vertices lie in the sets $A_j$s. This is possible on the event $\tilde{E}$ by construction and hence for any $v \in \mathcal{R}^*$,
\begin{align*}
\dfrac{\log \log S_v}{|v|} \geq \dfrac{1}{|v|} \log \log \sum_{k=1}^{|v|} \min_{u \in A_{k}} X_{e_u} \geq \dfrac{1}{|v|} \log \log \max_{k=1}^{|v|} \min_{u \in A_{k}} X_{e_u} &= \dfrac{1}{|v|} \log  \max_{k=1}^{|v|} \min_{u \in A_{k}} \log X_{e_u} \\
& = \dfrac{1}{|v|} \left( \log \dfrac{\delta\widetilde{W}}{(1-\delta)\beta} + |v| \log \dfrac{(1-\delta)}{\alpha} + o(1)\right),
\end{align*}
where $o(1)$ refers to a sequence of random variables which almost surely converges to $0$ on $\tilde{E}$. Therefore,
$$ \mathbb{P} \left[\sup_{\mathcal{R} \in \partial \mathbb{T}} \liminf_{v \in \mathcal{R}} \dfrac{\log \log S_v}{|v|} \geq   \log \dfrac{1-\delta}{\alpha} \bigg \rvert \mathbb{T} \text{ survives} \right] \geq \mathbb{P}(\tilde{E} \mid \mathbb{T} \text{ survives} ) = \dfrac{\tilde{q}}{\mathbb{P}(\mathbb{T} \text{ survives})} >0. $$
Since, conditioned on $\mathbb{T}$, the event inside the left-most term above lies in the tail $\sigma$-algebra generated by the i.i.d. displacement variables, we can invoke Kolmogorov $0-1$ law to conclude that almost surely conditioned on survival of the tree $\mathbb{T}$, 
 $$ \sup_{\mathcal{R} \in \partial \mathbb{T}} \liminf_{v \in \mathcal{R}} \dfrac{\log \log S_v}{|v|} \geq   \log \dfrac{1-\delta}{\alpha}. $$
 Taking $\delta \downarrow 0$, we complete the proof.
\end{proof} 
 
\begin{proof}[Proof of \Cref{speedlighttail}]
Taking logarithm on both sides on \cref{eq:deterministic} and using the fact that $\log L(\alpha^{-n}) \sim (-n/r) \log \alpha$ as $n \to \infty$, a fact which follows from the regular variation of $L$, we can conclude the upper bound :
$$ \limsup_{n \to \infty} \dfrac{\log M_n}{n} = -\dfrac{1}{r} \log \alpha, \; \text{ almost surely on survival of the tree }.$$
For the lower bound on burst speed, we use the same construction as in the proof of \cref{speedheavy}. \cref{minconv1} here changes to the following asymptotics by virtue of \cref{normordernew}.
\begin{equation}
\min_{v \in A_j} X_{e_v} \stackrel{a.s.}{\sim} L(\delta \log |B_j|) \stackrel{a.s.}{\sim} \left(\dfrac{\delta}{1-\delta}\right)^{1/r}\widetilde{W}^{1/r}L((1-\delta)^j \alpha^{-j}),\; \text{ on } \tilde{E}, 
\end{equation}
where we used the fact that $L$ is regularly varying at $\infty$ with index $1/r$. \cref{regulargeom} and \cref{cezaro} can now be applied to obtain the following.
$$  \sum_{j=1}^n \min_{v \in A_j} X_{e_v} \stackrel{a.s.}{\sim} \sum_{j=1}^n \left(\dfrac{\delta}{1-\delta}\right)^{1/r}\widetilde{W}^{1/r}L((1-\delta)^j \alpha^{-j})  \sim \left(\dfrac{\delta}{1-\delta}\right)^{1/r}\widetilde{W}^{1/r} \dfrac{L((1-\delta)^n \alpha^{-n})}{1-\alpha^{1/r}(1-\delta)^{-1/r}},  \; \text{ on } \tilde{E},$$
Hence as $|v| \to \infty$ via $\mathcal{R}^*$, we have
\begin{align*}
\dfrac{\log S_v}{|v|} \geq \dfrac{1}{|v|} \log \sum_{k=1}^{|v|} \min_{u \in A_{k}} X_{e_u} \stackrel{a.s.}{\sim} \dfrac{1}{|v|}\log L((1-\delta)^{|v|} \alpha^{-|v|}) \sim \dfrac{1}{r} \dfrac{\log \left((1-\delta)^{|v|} \alpha^{-|v|} \right)}{|v|} = \dfrac{1}{r} \log \left( \dfrac{1-\delta}{\alpha}\right).
\end{align*}
We can now finish th proof by the same argument as used in \cref{speedheavy}.
\end{proof}

\begin{proof}[Proof of \cref{speedverylight}]
This proof basically is a corollary of \cref{verylighttail:mainthm} as all the work had already been done there. The upper bound on the cloud speed follows directly from the assertion of \cref{verylighttail:mainthm}. Continuing with the notation introduced during its proof, note that for any ray $\mathcal{R}^*$ in the subtree formed by the subsets $\left\{A_{j, \eta} : j \geq 0\right\}$, we have the following almost surely on $E_{\eta}$.
$$ \liminf_{v \in \mathcal{R}^*} \dfrac{S_v}{\sum_{k=1}^{|v|} L(\alpha^{-k})} \geq \liminf_{v \in \mathcal{R}^*} \dfrac{\sum_{k=1}^{|v|} \min_{u \in A_{k,\eta}} X_{e_u}} {\sum_{k=1}^{|v|} L(\alpha^{-k})} \geq \liminf_{k \to \infty} \dfrac{ \min_{u \in A_{k,\eta}} X_{e_u}} { L(\alpha^{-k})} \geq \liminf_{k \to \infty} \dfrac{ L(\alpha^{-k}e^{-\kappa_k(\eta)})} { L(\alpha^{-k})} = \varphi(\eta). $$
The rest of the proof follows by arguments similar to those applied in the proof of \cref{speedheavy} and then taking $\eta \downarrow 0$. 
\end{proof}



\section{Appendix}
\label{app:prooflems}
In the Appendix we provide proofs for \cref{strongheavy1}, \cref{inhom} and \cref{normordernew}. 
We did not employ \cref{inhom} and \cref{strongheavy1} in its full generality, rather for a particularly engineered choice of the progeny sequence. Consider the progeny distribution $G$ of the original homogenous tree, satisfying \cref{davies_assump} with moment index $\alpha \in (0,1).$  We applied \cref{inhom} with $G_n=G_{(\psi_n)}$ where $\left\{\psi_n : n \geq 0\right\}$ is a sequence in $[0,1]$, bounded away from $0$. The validity of such choice is justified by \cref{thin1}, which is stated and proved here.

First we prove Lemma \ref{strongheavy1} which was used in the proof of Theorem \ref{lighttailmain}.

\begin{proof}[Proof of \cref{strongheavy1}]
Define $\alpha_{max} := \sup_{n \geq 0} \alpha_n <1$. By \cref{assump:inhom} and \cite[Lemma 1]{Davies} , there exists positive constants $0<c_1,c_2< \infty$ and
  $\lambda_0 \in (0,1)$ such that for all $\lambda \in (0, \lambda_0)$ the following holds.
  \begin{align}{\label{laplace:estimate:davies1}}
    c_1\lambda^{\alpha_n+\gamma(1/\lambda)} \leq 1- \Psi_n(\lambda) \leq  c_2\lambda^{\alpha_n-\gamma(1/\lambda)}, \; \forall \; n \geq 0,
  \end{align}
  where $\Psi_n$ is the Laplace transform of the distribution function $G_{n}$ and $\gamma$ is some function satisfying \cref{item:D1}, \cref{item:D2} and \cref{item:D3}. 
  
  The assumption~\ref{item:D3} and monotonicity of $\gamma$ guarantees that
  \begin{equation}{\label{summability}}
  \sum_{n \geq 1} \gamma\left( e^{c_3 e^{c_4 n}} \right) \leq \int_{0}^{\infty}  \gamma\left( e^{c_3 e^{c_4 x}} \right)  \, dx = \dfrac{1}{c_4} \int_{\log c_3}^{\infty} \gamma(e^{e^x})\, dx < \infty,
  \end{equation}
for any $c_3,c_4 >0$. By our assumption $\zeta :=\liminf_{n \to \infty} n^{-1} \log \log \zeta_{n} >0$, almost surely and hence \cref{summability}, applied for $c_3=(\alpha^{-1}-1)$ and $c_4=\zeta/2$, implies that 
$$ \sum_{n \geq 0} \gamma\left( {\zeta_n}^{(\alpha^{-1}-1)} \right) < \infty, \; \text{ almost surely}.$$

Set $\varepsilon_n = \max\left(\min \left(2\alpha_n^{-2}\gamma\left( {\zeta_n}^{(\alpha^{-1}-1)} \right) , \alpha_n^{-1},1 \right), (n+1)^{-2} \right) $ and $\upsilon_n = \varepsilon_n \alpha_n^{2}/2  $, for all $n \geq 0$. These choices guarantee that almost surely, $\sum_{n \geq 1} \varepsilon_n < \infty$ and for all $n \geq 0$,  $ \upsilon_n/\alpha_n - \varepsilon_n \upsilon_n < \varepsilon_n \alpha_n/2$. By construction, almost surely,  $\upsilon_n \geq \gamma\left( {\zeta_n}^{(\alpha^{-1}-1)} \right)$ for all large enough $n$ and hence, 
$$ \gamma \left(\zeta_{n}^{(\alpha_n^{-1}+\varepsilon_n)} \right) \leq \gamma \left(\zeta_{n}^{(\alpha_n^{-1}-\varepsilon_n)} \right) \leq \gamma_1\left( {\zeta_n}^{(\alpha^{-1}-1)} \right) \leq \upsilon_n.$$
The fact that $(\alpha_n^{-1}-\varepsilon_n) \log \zeta_n \stackrel{a.s.}{\longrightarrow} \infty$ guarantee that 
$ \zeta_{n}^{-\alpha_n^{-1}-\varepsilon_n}, \zeta_{n}^{-\alpha_n^{-1}+\varepsilon_n} < \lambda_0$, for all large enough $n$ almost surely. Combining this observation with \cref{laplace:estimate:davies1} yields the following. Almost surely, for all large enough $n$, 
\begin{equation}{\label{laplace1}}
 c_1\left(  \zeta_{n}^{-\alpha_n^{-1}+\varepsilon_n}\right)^{\alpha_n+\upsilon_n} \leq 1- \Psi_n\left(  \zeta_{n}^{-\alpha_n^{-1}+\varepsilon_n}\right),\;\; 1- \Psi_n\left(  \zeta_{n}^{-\alpha_n^{-1}-\varepsilon_n}\right) \leq   c_2\left(  \zeta_{n}^{-\alpha_n^{-1}-\varepsilon_n}\right)^{\alpha_n-\upsilon_n}.
\end{equation}


  With this preliminary set-up, we claim that, almost surely for all large enough $n$,
  \begin{equation}{\label{claim}}
  \left(\alpha_n^{-1}-\varepsilon_n \right) \log \zeta_{n} \leq \log \sum_{i=1}^{\zeta_{n}} L_{n,i} \leq \left(\alpha_n^{-1}+\varepsilon_n \right) \log \zeta_{n}.
  \end{equation}
  Before proving the claim, let us first see how this claim helps us to achieve our target. Denote by $E_n$ the intersection of $(\zeta_n >1)$ and the event in \cref{claim}. Note that, \cref{claim} implies $N:= \sup \left\{n \geq 0 : \mathbbm{1}_{E_n}=1\right\} < \infty$, almost surely. Therefore, 
  $$ \sum_{n \geq N+1} \bigg \rvert \alpha_n \dfrac{\log \sum_{i=1}^{\zeta_{n}} L_{n,i}}{\log \zeta_{n}} - 1\bigg \rvert \leq \sum_{n \geq N+1} \alpha_n\varepsilon_n  < \infty,$$
  almost surely, since $\alpha_n \leq \alpha_0$, $\sum_{n \geq 1} \varepsilon_n < \infty$ with probability $1$. A standard analysis fact now yields that
  $$ 0 < \prod_{n \geq N+1} \left( \alpha_n \dfrac{\log \sum_{i=1}^{\zeta_{n}} L_{n,i}}{\log \zeta_{n}} \right)< \infty, \; \text{almost surely }.$$

  Let us now proceed to prove the lower bound in \cref{claim}. The following series of
  inequalities hold true for large enough $n$, almost surely. 
  \begin{align}
    \mathbb{P} \left[ \log \sum_{i=1}^{\zeta_{n}} L_{n,i} \leq \left(\alpha_n^{-1}-\varepsilon_n \right) \log \zeta_{n} \Bigg \rvert \mathcal{F}_{n}\right]
    & =  \mathbb{P} \left[ \exp \left( -\zeta_{n}^{-\alpha_n^{-1}+\varepsilon_n}\sum_{i=1}^{\zeta_{n}} L_{n,i} \right) \geq 1/e \Bigg \rvert \mathcal{F}_{n} \right] \nonumber \\
    & \leq e\left(\Psi_n \left( \zeta_{n}^{-\alpha_n^{-1}+\varepsilon_n}\right)\right)^{\zeta_{n}}\nonumber \\
    & \leq   \exp \left( 1+\zeta_{n}\log \left( 1- c_1\zeta_{n}^{(-\alpha_n^{-1}+\varepsilon_n)(\alpha_n+\upsilon_n))}\right)\right) \nonumber \\
    & \leq   \exp \left(1+ \zeta_{n}\log \left( 1- c_1\zeta_{n}^{-1+\varepsilon_n \alpha_n/2}\right)\right) 
      \leq  \exp \left(1-c_1\zeta_{n}^{\varepsilon_n \alpha_n/2} \right) \label{lower:lemma451} ,
  \end{align}
  where the penultimate inequality follows from the fact that
  $\varepsilon_n \alpha_n/2 >\upsilon_n/\alpha_n - \varepsilon_n \upsilon_n$ and the last one by the bound
  $-\log(1-x)\geq x$. Observe that, almost surely,
  \begin{equation}{\label{growth}}
  \liminf_{n \to \infty} n^{-1}\log \zeta_{n}^{\varepsilon_n \alpha_n/2} = \dfrac{1}{2}\liminf_{n \to \infty} n^{-1} \epsilon_n \alpha_n\log \zeta_{n} \geq  \dfrac{\alpha}{2}\liminf_{n \to \infty} n^{-3}   \log \zeta_{n} = \infty,
  \end{equation}
  and hence the last expression in (\ref{lower:lemma451}) is summable almost surely. Applying Levy's extension of Borel-Cantelli Lemma  we conclude the proof of the lower bound.
  

 To prove the upper bound in \cref{claim}, we establish similar kind of inequalities which hold almost surely for large enough $n$. 
  \begin{align}
   \mathbb{P} \left[ \log \sum_{i=1}^{\zeta_{n}} L_{n,i} \geq \left(\alpha_n^{-1}+\varepsilon \right) \log \zeta_{n} \Bigg \rvert \mathcal{F}_{n} \right]
    & =  \mathbb{P} \left[ 1-\exp \left( -\zeta_{n}^{-\alpha_n^{-1}-\varepsilon_n}\sum_{i=1}^{\zeta_{n}} L_{n,i} \right) \geq 1- 1/e   \Bigg \rvert \mathcal{F}_{n}  \right] \nonumber \\
    & \leq (1-1/e)^{-1} \left[1-\left(\Psi_n \left( \zeta_{n}^{-\alpha_n^{-1}-\varepsilon_n}\right)\right)^{\zeta_{n}} \right] \nonumber \\
    & \leq (1-1/e)^{-1} \left[1-\left(1- c_2 \left( \zeta_{n}^{-\alpha_n^{-1}-\varepsilon_n}\right)^{\alpha_n-\upsilon_n}\right)^{\zeta_{n}} \right] \nonumber \\
    & \leq (1-1/e)^{-1} \left[1-\left(1- c_2  \zeta_{n}^{-1-\varepsilon_n \alpha_n/2}\right)^{\zeta_{n}} \right] \label{last1} \\
    & \leq (1-1/e)^{-1} c_2 \zeta_{n}^{-\varepsilon_n \alpha_n/2}, \label{lastestim1}
  \end{align}
  where (\ref{last1}) is implied by the fact that  $\upsilon_n/\alpha_n + \varepsilon_n \upsilon_n < \varepsilon_n \alpha_n/2$, and (\ref{lastestim1}) uses the  estimate $1-\left(1- x\right)^\kappa \leq \kappa x$ for all $x\in(0,1]$ and
      $\kappa>0$.  Almost sure Summability of the expression in (\ref{lastestim1}) follows from \cref{growth}.  Applying Levy's extension of Borel-Cantelli Lemma we conclude the proof of the upper bound.
\end{proof}

\begin{proof}[Proof of \cref{inhom}] 
\cref{assump:inhom} guarantees that $\alpha_n  \leq \alpha_{max} < 1$, for all $n \geq 0$. By \Cref{stochdom}, it is  possible to get hold of a distribution function $G_{max}$, supported on non-negative integers, that satisfies \cref{davies_assump} with moment index $\alpha^{*} \in (\alpha_{max},1)$ and $G_{max} \geq G_n$, pointwise, for all $n \geq 0$. Therefore, we can get a coupling measure $\nu_n$ on $\mathbb{Z}^2$ such that 
$$ (Z_{1,1},Z_{1,2}) \sim \nu_n \Rightarrow Z_{1,1} \sim G_{max}, Z_{1,2} \sim G_n, Z_{1,1} \leq Z_{1,2} \text{ almost surely }.$$

Consider the following triangular array of pairs of random variables $\left\{(Z_{n,i,1},Z_{n,i,2}) : n \geq 0, i \geq 1\right\}$ where $(Z_{n,i,1},Z_{n,i,2}) \sim \nu_n$, independent of every other pair in this array. Define, 
$$ Z_{0,j}:=1, ; Z_{n,j}:= \sum_{i=1}^{Z_{n-1,j}} Z_{n-1,i,j}, \; \forall \; n \geq 1, \; j=1,2.$$
The following observations are immediate from the construction.
\begin{enumerate}
\item $\left\{Z_{n,1} : n \geq 0\right\}$ is the generation sizes of a homogeneous branching tree with progeny distribution $G_{max}$, starting with one particle in $0$-th generation. Since $G_{max}$ satisfies \cref{davies_assump} with moment index $\alpha_{max}$, we can apply \cref{davies} to conclude that there exists a non-negative non-degenerate random variable  
$W_1$ such that $n^{-1} \log \log Z_{n,1}$ converges almost surely to  $ -\log \alpha_{max} >0$, on the event $(W_1>0)$. Note that $\mathbb{P}(W_1>0)>0$.

\item The sequence $\left\{Z_{n,2} : n \geq 0\right\}$ is the generation sizes of an in-homogeneous branching process with progeny distribution for the particles in the $n$-th generation being $G_{n}$. In other words, $ \left\{Z_{n,2} : n \geq 0\right\}  \stackrel{d}{=} \left\{Z_n : n \geq 0\right\}$, as defined in the statement of \cref{inhom} and hence it is enough to prove the assertion of the theorem for  $ \left\{Z_{n,2} : n \geq 0\right\}$.

\item  $(Z_{n,i,1},Z_{n,i,2}) \sim \nu_n \Rightarrow Z_{n,i,1} \leq Z_{n,i,2}, \; \forall \; i,n \Rightarrow Z_{n,1} \leq Z_{n,2}, \; \forall \; n \geq 0.$ Therefore, 
$$ \liminf_{n \to \infty} \dfrac{1}{n} \log \log Z_{n,2} \geq \liminf_{n \to \infty} \dfrac{1}{n} \log \log Z_{n,1} = - \log \alpha_{max} > 0, \; \text{ on } (W_1>0).$$ 
\end{enumerate}

We can now apply \cref{strongheavy1} with $\mathcal{F}_n$ being the $\sigma$-algebra generated by the collection of random variables $\left\{Z_{k,i,1},Z_{k,i,2} : 0 \leq k \leq n-1\right\}$ with $\zeta_n=Z_{n,2}$ and $L_{n,i} =Z_{n,i,2}$, for all $n \geq 0$. We conclude that there exists a non-negative integer valued random variable $N$, satisfying $\mathbb{P}(N = \infty, W_1 >0)=0$, such that the following occurs
$$ \left(\prod_{m=N+1}^{n} \alpha_m \right) \log Z_{n+1,2}(\omega) = \log Z_{N+1,2}(\omega)\prod_{m=N+1}^{n} \left(  \dfrac{\alpha_m \log Z_{m+1,2}(\omega)}{\log Z_{m,2}(\omega)}\right) \longrightarrow W_2(\omega) \in (0, \infty), \; \text{ as } n \to \infty,$$
for almost all $\omega \in (W_1 >0)$. Here $W_2$ is some positive almost surely finite random variable. This implies our final assertion after we take $E:=(W_1>0)$ and define $W^{*}$ as follows.
$$ W^*(\omega) = \begin{cases}
W_2(\omega) \prod_{m=0}^{N(\omega)} \alpha_m, & \text{ if } N(\omega) < \infty, W_1(\omega) >0, \\
1, & \text{ otherwise }.
\end{cases}$$
\end{proof}

\begin{lemma}
  \label{thin1}
  Let $G$ be a distribution function, supported on non-negative real line but not necessarily on the set of non-negative integers, satisfying \cref{item:D1,item:D2,item:D3,item:D4} with moment index $\alpha \in (0,1)$.
   Fix $\delta_0 \in (0,1]$. Then there exists $\gamma_1 : \mathbb{R}^{+} \to \mathbb{R}^{+}$, satisfying \cref{item:D1}, \cref{item:D2} and \cref{item:D3}, such that for some $x_1>1$ and for any $x \geq x_1$, we have
   $$ x^{-\gamma_1(x)} \leq x^{\alpha/\delta}(1-G_{(\delta)}(x)) \leq x^{\gamma_1(x)}, \; \forall \; \delta \in [\delta_0,1].$$
   In particular, $G_{(\delta_0)}$ also satisfies Assumption~\ref{davies_assump} with
  moment index $\alpha/\delta_0$, provided $\delta_0 > \alpha$.
 \end{lemma}

\begin{proof}[ Proof of Lemma~\ref{thin1}]
We have, by assumption, $\gamma : \mathbb{R}^{+} \to \mathbb{R}^{+}$, satisfying \cref{item:D1}, \ref{item:D2} and \ref{item:D3}, such that 
$$ x^{-\gamma(x)} \leq x^{\alpha}(1-G(x)) \leq x^{\gamma(x)}, \; \forall \; x \geq x_0,$$
for some $x_0 \in (1,\infty)$. Fix $\delta \in [\delta_0,1].$ For any $x \geq x_0 >1$, we have $x^{1/\delta} > x_0^{1/\delta} \geq x_0$ and therefore can write the following. 
\begin{align}{\label{upper}}
1-G_{(\delta)}(x) = \mathbb{P}\left(\lfloor Z^{\delta} \rfloor > x \right) \leq \mathbb{P} \left( Z > x^{1/\delta}\right) = 1-G(x^{1/\delta}) \leq x^{-\alpha/\delta}x^{\gamma(x^{1/\delta})/\delta} \leq x^{-\alpha/\delta}x^{\gamma(x)/\delta}, 
\end{align}
where the last inequality follows from non-monotonicity of $\gamma$ and the fact that $\delta \leq 1$.
Monotonicity of $\gamma$ also implies that
$$ 1 > \eta := \inf_{\delta \in [\delta_0,1]} \inf_{x \geq 1} \left( (x+1)/x\right)^{-\alpha/\delta - \gamma(x)/\delta} \geq   \inf_{\delta \in [\delta_0,1]} \inf_{x \geq 1} \left( (x+1)/x\right)^{-\alpha/\delta - \gamma(1)/\delta}= 2^{-\alpha/\delta_0 - \gamma(1)/\delta_0} >0, $$
and hence for all $x \geq  x_0$,
\begin{align}
1-G_{(\delta)}(x) = \mathbb{P}\left(\lfloor Z^{\delta} \rfloor > x \right) \geq \mathbb{P} \left( Z^{\delta} > x+1\right) = 1-G((x+1)^{1/\delta}) &\geq (x+1)^{-\alpha/\delta}(x+1)^{-\gamma((x+1)^{1/\delta})/\delta} \nonumber \\
& \geq (x+1)^{-\alpha/\delta}x^{-\gamma(x)/\delta}, \label{nondec} \\
& \geq \eta x^{-\alpha/\delta}x^{-\gamma(x)/\delta}, \label{lower} 
\end{align}
where \cref{nondec} uses monotonicity of the map $y \mapsto y^{\gamma(y)}$. Combining \cref{upper} and \cref{lower}, we can write the following for any $x \geq x_0$ and $\delta \in [\delta_0,1]$.
\begin{equation}{\label{combin}}
 x^{-\gamma(x)/\delta_0}\eta \leq x^{\alpha/\delta}\left(1-G_{(\delta)}(x) \right)  \leq  x^{\gamma(x)/\delta} \leq  x^{\gamma(x)/\delta_0}\eta^{-1}.
\end{equation}
 Defining
  $\gamma_1 : (0, \infty) \to (0,\infty)$ to be
  $$
  \gamma_1(x) := \begin{cases} \dfrac{\gamma(x)}{\delta_0}  -\dfrac{\log \eta}{\log x}, \;\; \text{ if } x \geq x_0, \\
    \gamma_1(x_0), \;\; \hspace{0.42 in} \text{ if } x < x_0, \end{cases}
  $$
  it follows immediately that for all $x \geq x_0$ and $\delta \in [\delta_0,1]$, 
  $$ x^{-\alpha/\delta-\gamma_1(x)} \leq  1-G_{(\delta)}(x) \leq x^{-\alpha/\delta + \gamma_1(x)} .$$
  It is now enough to show that \cref{item:D1}, \cref{item:D2} and \cref{item:D3} are also satisfied if we replace $\gamma$ by $\gamma_1$. The monotonicity of $\gamma$ implies that $\gamma_1$ is also non-increasing. On the other hand, for $x \geq x_0$,
  $x^{\gamma_1(x)} = \eta^{-1} x^{\gamma(x)/\delta_0}$ which guarantees that $x \mapsto x^{\gamma_1(x)}$ is
  non-decreasing. Furthermore,
  $$
  \int_{\log \log x_0}^{\infty} \gamma_1(e^{e^x}) \,dx =\dfrac{1}{\delta_0}\int_{\log \log x_0}^{\infty}
  \gamma(e^{e^x}) \,dx -\log \eta \int_{\log \log x_0}^{\infty}  \exp(-x) \,dx < \infty,
  $$
  as $\int_{0}^{\infty} \gamma(e^{e^x}) \,dx < \infty$. Therefore,
  $\int_{0}^{\infty} \gamma_1(e^{e^x}) \,dx < \infty$ and  this proves the lemma.
\end{proof}

\begin{lemma}{\label{stochdom}}
Let $\left\{G_n : n \geq 0\right\}$ satisfies \Cref{assump:inhom}. Then for any $\beta \in \left( \sup_{n \geq 0} \alpha_n,1\right)$, there exists $G_{max}$, a distribution function supported on the set of non-negative integers, such that $G_{max}$ satisfies \Cref{davies_assump} with moment index $\beta $  and $G_{max} \geq G_n$, pointwise, for all $n \geq 0$.
\end{lemma}
\begin{proof}
Let $\alpha := \sup_{n \geq 0} \alpha_n <1$ and fix $\beta \in \left( \alpha,1\right)$. Get $x_1$ large enough such that $ \beta > \alpha+\gamma(x)$, for all $x \geq x_1$.  This is possible since $\gamma$ satisfies \Cref{item:D1} and \Cref{item:D3}; and hence $\gamma(x) \downarrow 0$ as $x \uparrow \infty$. Set $\tilde{x}=  x_0 \vee x_1 \vee 1$ and define
$$ G^*(x) := \begin{cases}
\sup_{n \geq 0} G_n(x), & \text{ if } x < \tilde{x}, \\
1- x^{-\beta}, & \text{ if } x \geq \tilde{x}.
\end{cases} $$
Clearly, $G^*(0)=0$ and $G^*(\infty)=1$. Right continuity of $G^*$ on $[\tilde{x}, \infty)$ is obvious whereas on $(-\infty,\tilde{x})$ it is guaranteed by the fact that $G_n$ is supported on the set of integers for all $n \geq 0$. Moreover,  
$$ G^*(\tilde{x}-) \leq \sup_{n \geq 0} G_n(\tilde{x}) \leq \sup_{n \geq 0} \left( 1-\tilde{x}^{-\alpha_n-\gamma(\tilde{x})}\right)  \leq 1- \tilde{x}^{-\alpha-\gamma(\tilde{x})} \leq 1- \tilde{x}^{-\beta} = G^*(\tilde{x}),$$
implying that $G^*$ is indeed a distribution function, supported on the non-negative real line. It is obvious from the definition that $G^*$  satisfies \Cref{item:D4} with moment index $\beta$. Finally, for any $x \geq \tilde{x}$, 
$$ 1-G^*(x) = x^{-\beta} \leq  x^{-\alpha_n - \gamma(x)} \leq 1-G_n(x),$$
guaranteeing that $G^* \geq G_n$ for all $n \geq 0$, pointwise. Let $G_{max}$ be the distribution function of $\lfloor Z \rfloor$ where $Z \sim G^*$. $G_{max}$ is clearly supported on the set of non-negative integers; $G_{max} \geq G^* \geq G_n$ for all $n \geq 0$ and \Cref{thin1} guarantees that $G_{max}$ satisfies \Cref{davies_assump} with moment index $\beta$. This completes the proof.
\end{proof}

\begin{proof}[Proof of \cref{normordernew}]

Consider the case of $\psi_n$ being bounded away from $1$. Fix any positive $0<\varepsilon < 1,$ small enough. Let us first prove the lower bound of the limit.   
\begin{align*}
\mathbb{P} \left[ G_{l_n:\zeta_n} \leq L\left((1-\varepsilon)\psi_n\log \zeta_n\right) \Big \rvert \mathcal{F}_n \right]& \leq  \mathbb{P} \left[ \sum_{i=1}^{\zeta_n} \mathbbm{1} \left( G_{n,i} \geq L\left((1-\varepsilon)\psi_n\log \zeta_n\right)\right) \leq l_n \Bigg \rvert \mathcal{F}_n  \right] \\
&= \mathbb{P} \left[ \text{Binomial} \left(\zeta_n,p_n \right) \leq l_n \Big \rvert \mathcal{F}_n \right]
\end{align*}
where $p_n := \bar{F} \left(L\left((1-\varepsilon)\psi_n\log \zeta_n \right) \right) = \exp\left(-K \circ L\left((1-\varepsilon)\psi_n\log \zeta_n \right) \right)$. By assumption on $F$, we have $K(L(x)) \sim x$ as $x \to \infty$. Therefore, almost surely, for large enough $n$, we can say $\exp\left(-(1-\varepsilon/2)\psi_n\log \zeta_n \right) \leq p_n \leq \exp\left(-(1-2\varepsilon)\psi_n\log \zeta_n \right),$ and hence
 $l_n \leq \zeta_n^{1-\psi_n(1-\varepsilon/4)} \leq \zeta_n^{1-\psi_n(1-\varepsilon/2)} \leq \zeta_np_n \leq \zeta_n^{1-\psi_n(1-2\varepsilon)}.$
Using Chebyshev's Inequality, we now obtain the following almost surely for all large enough $n$. 
\begin{align*}
\mathbb{P} \left[ G_{l_n:\zeta_n} \leq L\left((1-\varepsilon)\psi_n\log \zeta_n \right) \Big \rvert \mathcal{F}_n \right]
& \leq  \dfrac{\zeta_np_n(1-p_n)}{(\zeta_np_n-l_n)^2} \leq  \dfrac{\zeta_n^{1-\psi_n(1-2\varepsilon)}}{\left(\zeta_n^{1-\psi_n(1-\varepsilon/2)} - \zeta_n^{1-\psi_n(1-\varepsilon/4)} \right)^2} \leq  4\zeta_n^{-1+\psi_n(1+\varepsilon)}. 
\end{align*}
The fact that almost surely $\zeta_n^{\psi_n} = e^{\psi_n \log \zeta_n} \to \infty$ was crucial in deriving the last inequality in the above line. The last term being summable almost surely (since $\liminf_{n \to \infty} n^{-1} \log \zeta_n > 0$) for small enough $\varepsilon$, we use Levy's extension of Borel-Cantelli Lemma to conclude that for small enough $\varepsilon$,
\begin{equation}{\label{firstlim}}
 \liminf_{n \to \infty} \dfrac{ G_{l_n:\zeta_n}}{L\left((1-\varepsilon)\psi_n\log \zeta_n \right)} \geq 1, \; \text{ almost surely }.
\end{equation}
Since $\psi_n \log \zeta_n \sim \log l_n \to \infty$, almost surely, we have
\begin{equation}{\label{secondlim}}
\lim_{\varepsilon \downarrow 0} \lim_{n \to \infty} \dfrac{L\left(\psi_n\log \zeta_n \right)}{L\left((1-\varepsilon)\psi_n\log \zeta_n \right)} = 1, \; \text{ almost surely }.
\end{equation}
Here we have made use of the assumption that $\lim_{t \to 1} \lim_{x \to \infty} L(tx)/L(x)=1.$ Combining \cref{firstlim} and \cref{secondlim},  we conclude this case after taking $\varepsilon \downarrow 0$. Similarly for the upper bound, we have the following.
\begin{align*}
\mathbb{P} \left[ G_{l_n:\zeta_n} \geq L\left((1+\varepsilon)\psi_n\log \zeta_n \right) \Big \rvert \mathcal{F}_n \right]& \leq  \mathbb{P} \left[ \sum_{i=1}^{\zeta_n} \mathbbm{1} \left( G_{n,i} \geq L\left((1+\varepsilon)\psi_n\log \zeta_n \right) \right) \geq l_n \Big \rvert \mathcal{F}_n  \right] \\
&= \mathbb{P} \left[ \text{Binomial} \left(\zeta_n,q_n \right) \geq l_n  \Big \rvert \mathcal{F}_n \right]
\end{align*}
where $q_n := \bar{F} \left(L\left((1+\varepsilon)\psi_n\log \zeta_n \right) \right) = \exp\left(-K \circ L\left((1+\varepsilon)\psi_n\log \zeta_n \right) \right)$. Therefore, almost surely, for large enough $n$, we can say $\exp\left(-(1+2\varepsilon)\psi_n\log \zeta_n \right) \leq q_n \leq \exp\left(-(1+\varepsilon/2)\psi_n\log \zeta_n \right),$ and hence
 $l_n \geq \zeta_n^{1-\psi_n(1+\varepsilon/4)} \geq \zeta_n^{1-\psi_n(1+\varepsilon/2)} \geq \zeta_nq_n 
 .$
Using Chebyshev's Inequality again, we  obtain the following almost surely for all large enough $n$. 
\begin{align*}
\mathbb{P} \left[ G_{l_n:\zeta_n} \geq  L\left((1+\varepsilon)\psi_n\log \zeta_n \right) \Big \rvert \mathcal{F}_n \right]
& \leq  \dfrac{\zeta_nq_n(1-q_n)}{(l_n-\zeta_nq_n)^2} \leq  \dfrac{\zeta_n^{1-\psi_n(1+\varepsilon/2)}}{\left( \zeta_n^{1-\psi_n(1+\varepsilon/4)} -\zeta_n^{1-\psi_n(1+\varepsilon/2)}  \right)^2} \leq 4 \zeta_n^{-1+\psi_n}. 
\end{align*}
 The last term being summable almost surely, we use arguments similar to what were used for lower bound and complete the proof for the upper bound.

For $l_n \equiv \psi_n \equiv 1$, we observe that, for any $\kappa \in (0,1)$, we have $G_{l_n:\zeta_n} \geq G_{\lfloor \zeta_n^{1-\kappa} \rfloor :\zeta_n}$, for all large enough $n$, almost surely. Hence,
$$ \liminf_{n \to \infty} \dfrac{G_{l_n:\zeta_n}}{L\left( \log \zeta_n\right)} \geq 
\liminf_{n \to \infty} \dfrac{G_{\lfloor \zeta_n^{1-\kappa} \rfloor :\zeta_n}}{L\left( \log \zeta_n\right)}  = \liminf_{n \to \infty} \dfrac{L(\kappa \log \zeta_n)}{L(\log \zeta_n)}, \;\; \text{almost surely}.$$ The lower bound then follows from  taking $\kappa \uparrow 1$. On the otherhand, for any $\varepsilon >0$, 
\begin{align*}
\mathbb{P} \left[ G_{1:\zeta_n} \geq L\left((1+\varepsilon)\log \zeta_n \right) \Big \rvert \mathcal{F}_n \right]& \leq  \mathbb{P} \left[ \sum_{i=1}^{\zeta_n} \mathbbm{1} \left( G_{n,i} \geq L\left((1+\varepsilon)\log \zeta_n \right) \right) \geq 1 \Big \rvert \mathcal{F}_n  \right] \\
& \leq \zeta_n \bar{F} \left(L\left((1+\varepsilon)\log \zeta_n \right) \right) \\
& = \zeta_n \exp\left(-K \circ L\left((1+\varepsilon)\log \zeta_n \right) \right) \leq \zeta_n^{-\varepsilon/2},
\end{align*}
for large enough $n$, almost surely. The rest of the argument follows similarly as the previous one.
\end{proof}

\begin{lemma}
  \label{normordernewreg}
  Consider the same set-up as in \cref{normordernew}, but assume that the distribution function $F$ satisfies \cref{assumption:displacement_heavy}. Then,
     \begin{align*}
        \frac{\log G_{l_{n}:\zeta_{n}}}{\psi_n \log \zeta_{n}} \stackrel{a.s.}{\longrightarrow} \dfrac{1}{\beta}, \; \text{ as } n \to \infty.\;
     \end{align*}
  \end{lemma}

  \begin{proof}[Proof of \cref{normordernewreg}] 
  Note that, conditioned on $\mathcal{F}_n$, the random variable $\log G_{n,i}$ have distribution function given by $\widetilde{F}(x) = \mathbb{P}(\log G_{n,1} \leq x) = F(e^x)$, for all $x$. In light of \Cref{Fsat}, it is enough to prove that $\widetilde{K} := -\log (1-\widetilde{F})$ is regularly varying at $\infty$ with index $r=1$ and $\widetilde{L}(x) \sim \beta^{-1}x$ as $x \to \infty$, where $\widetilde{L}$ is the left-continuous inverse of $\widetilde{K}$.
  
  Since $1-F$ is regularly varying at $\infty$ with index $-\beta <0$, we can apply \cite[Proposition 0.8(i)]{Resnick} to conclude that $\log (1-F(x)) \sim -\beta \log x$, as $x \to \infty$. Therefore, for any $t>0$,
  $$ \dfrac{\widetilde{K}(tx)}{\widetilde{K}(x)} = \dfrac{- \log (1-F(e^{tx}))}{-\log (1-F(e^x))} \sim \dfrac{\beta tx}{\beta x} \sim t, \;\; \text{ as } x \to \infty,$$
  and hence $\widetilde{K}$ is regularly varying at $\infty$ with index $1$. On the otherhand,
  $$ \widetilde{K}(x) = - \log (1-F(e^x)) \sim \beta x, \; \text{ as } x \to \infty.$$
  As mentioned in \Cref{Fsat}, we have $\widetilde{K}(\widetilde{L}(x)) \sim x$ as $x \to \infty$; hence $x \sim \widetilde{K}(\widetilde{L}(x)) \sim  \beta \widetilde{L}(x)$. This completes the proof.
  \end{proof}

\begin{lemma}
  \label{regulargeom}
  If $h$ is regularly varying at $\infty$ with index $\rho>0$ and $a \in(0,1),$ then
  \begin{align*}
    \lim _{n \rightarrow \infty} \frac{1}{h\left(a^{-n}\right)} \sum_{i=1}^{n}
    h\left(a^{-i}\right)=\frac{1}{1-a^{\rho}}.
  \end{align*}
\end{lemma}

\begin{proof}
  For each fixed $m,$ we have, by regular variation
  \begin{align*}
    \lim _{n \rightarrow \infty} \frac{1}{h\left(a^{-n}\right)} \sum_{i=n-m}^{n} h\left(a^{-i}\right)
    &=\lim _{n \rightarrow \infty} \frac{1}{h\left(a^{-n}\right)} \sum_{i=n-m}^{n} h\left(a^{n-i} a^{-n}\right) \\
    &=\lim _{n \rightarrow \infty} \frac{1}{h\left(a^{-n}\right)} \sum_{i=0}^{m} h\left(a^{i} a^{-n}\right)=\sum_{i=0}^{m} a^{i \rho} \; .
  \end{align*}
  The last sum tends to $\left(1-a^{\rho}\right)^{-1}$ as $m$ tends to infinity, thus the lemma will
  be proved if we check that
  \begin{equation}
    \label{toshow}
    \lim _{m \rightarrow \infty} \lim _{n \rightarrow \infty} \frac{1}{h\left(a^{-n}\right)} \sum_{i=1}^{n-m} h\left(a^{-i}\right)=0.
  \end{equation}
  We have $h(x)=\ell(x) x^{\rho}$ with $\ell$ being a slowly varying function. Using \cite[Theorem 0.6]{Resnick}, for $x >0$,
  $\ell(x) = c(x)\exp \left( \int_{1}^x t^{-1} \xi (t) \, dt \right)$, where
  $\lim_{x \to \infty} c(x) = c \in (0,\infty)$ and $\lim_{x \to \infty} \xi(x) = 0.$ Hence, for
  every $\varepsilon>0,$ for large enough $x<y$, 
  \begin{align*}
    \frac{x^{\varepsilon}\ell(x)}{y^{\varepsilon}\ell(y)} = \frac{c(x)}{c(y)} \exp \left(-
    \int_{x}^y t^{-1} (\xi (t)+\varepsilon) \, dt \right) \leq \frac{c(x)}{c(y)} \leq \Lambda,
  \end{align*}
  for some finite constant $\Lambda$.  Thus, we can find finite constant $A_{\varepsilon} \geq 1$ such that
  for large enough $n$,
  \begin{align*}
    \sup _{1 \leq i \leq n} \frac{a^{\varepsilon(i-n)} \ell\left(a^{i-n}\right)}{a^{-\varepsilon n}
    \ell\left(a^{-n}\right)} \leq \sup _{1 \leq x \leq a^{-n}} \frac{x^{\varepsilon}
    \ell(x)}{a^{-\varepsilon n} \ell\left(a^{-n}\right)} \leq \frac{\sup _{1 \leq x \leq A_{\varepsilon}}
    x^{\varepsilon} \ell(x)}{a^{-\varepsilon n} \ell\left(a^{-n}\right)}+\Lambda.
  \end{align*}
  Thus, for $\varepsilon \in(0, \rho)$ and large $n$,
  \begin{align*}
    \sup _{1 \leq i \leq n} \frac{a^{\varepsilon(i-n)} \ell\left(a^{i-n}\right)}{a^{-\varepsilon n}
    \ell\left(a^{-n}\right)} \leq 2\Lambda,
  \end{align*}
      and
  \begin{align*}
    \frac{1}{h\left(a^{-n}\right)} \sum_{i=1}^{n-m} h\left(a^{-i}\right)
    & = \sum_{i=m}^{n-1} \frac{h\left(a^{-n} a^{i}\right)}{h\left(a^{-n}\right)}=\sum_{i=m}^{n-1}
      a^{(\rho-\varepsilon) i} \frac{a^{\varepsilon(i-n)} \ell\left(a^{i-n}\right)}{a^{-\varepsilon n} \ell\left(a^{-n}\right)} \\
    & \leq 2\Lambda \sum_{i=m}^{\infty} a^{(\rho-\varepsilon) i}=\frac{2\Lambda  a^{m(\rho-\varepsilon)}}{1-a^{\rho-\varepsilon}} \; .
  \end{align*}
  Taking $n,m \to \infty$, (\ref{toshow}) follows.
\end{proof}

\begin{lemma}{\label{cezaro}}
Take two sequence $\left\{a_n\right\}_{n \geq 1}$ and $\left\{b_n\right\}_{n \geq 1}$ of  real numbers such that  $\sum_{k=1}^n b_k \uparrow \infty.$ Then
$$ \liminf_{n \to \infty} a_n \leq \liminf_{n \to \infty} \dfrac{\sum_{i=1}^n a_ib_i}{\sum_{i=1}^n b_i} \leq \limsup_{n \to \infty} \dfrac{\sum_{i=1}^n a_ib_i}{\sum_{i=1}^n b_i} \leq \limsup_{n \to \infty} a_n.$$
\end{lemma}

\begin{proof}
The proof is a straightforward analysis exercise. 
\end{proof}

\begin{lemma}
Suppose $L$ is non-decreasing and slowly varying at $\infty$, with $L(\infty)=\infty$. Take two sequence $\left\{a_n\right\}_{n \geq 1}$ and $\left\{b_n\right\}_{n \geq 1}$ of non-negative real numbers such that $b_n \uparrow \infty$ as $n \to \infty$ and $a_n/b_n \in [\varepsilon, 1/\varepsilon]$ for some $\varepsilon >0$ and for all $n \geq 1$. Then
$$ \dfrac{\sum_{i=1}^n L(a_ib_i)}{\sum_{i=1}^n L(b_i)} \longrightarrow 1, \; \text{ as } n \to \infty.$$
\end{lemma}

\begin{proof}
Without loss of generality we can assume that $b_n >0$ and $L(b_n) \geq L(b_1) >0$ for all $n \geq 1$. Clearly, we have $\sum_{n \geq 1} L(b_n)=\infty$. In light of \cref{cezaro}, we only need to show that $L(a_nb_n)/L(b_n) \longrightarrow 1$. Monotonicity of $L$ now guarantees that $L(\varepsilon b_n) \leq L(a_nb_n) \leq L(b_n/\varepsilon)$, whereas slowly varying property implies that $L(\varepsilon b_n), L(b_n/\varepsilon) \sim L(b_n)$, as $n \to \infty$. This completes the proof.

\end{proof}




\end{document}